\documentclass[11pt,a4paper,twoside,reqno,final]{amsart}
\pdfoutput=1

\usepackage{fixltx2e} %

\usepackage[usenames,dvipsnames]{xcolor}
\usepackage{fancyhdr}
\usepackage{amsmath,amsfonts,amsbsy,amsgen,amscd,mathrsfs,amssymb,amsthm}
\usepackage{url}

\usepackage[font=small,margin=0.25in,labelfont={sc},labelsep={colon}]{caption}

\usepackage{eurosym}
\usepackage{tikz}
\usetikzlibrary{matrix,arrows}
\usepackage{subfig}

\usepackage{microtype}
\usepackage{enumitem}

\usepackage{cite}

\definecolor{dark-gray}{gray}{0.3}
\definecolor{dkgray}{rgb}{.4,.4,.4}
\definecolor{dkblue}{rgb}{0,0,.5}
\definecolor{medblue}{rgb}{0,0,.75}
\definecolor{rust}{rgb}{0.5,0.1,0.1}

\usepackage[colorlinks=true]{hyperref}
\hypersetup{linkcolor=dkblue}          %
\hypersetup{citecolor=rust}        %
\hypersetup{urlcolor=rust}           %

\usepackage{graphicx}
\usepackage{booktabs,longtable,tabu} %
\usepackage{multirow} %

\usepackage{float}

\usepackage[T1]{fontenc}

\usepackage{fourier}
\usepackage{charter}

\usepackage{bm} %

\graphicspath{{figures/}}

\usepackage{pdfsync}

\newtheorem{theorem}{Theorem}[section]
\newtheorem{lemma}[theorem]{Lemma}

\newtheorem{proposition}[theorem]{Proposition}

\newtheorem{corollary}[theorem]{Corollary}

\theoremstyle{definition}

\newtheorem{definition}[theorem]{Definition}

\newtheorem{remark}[theorem]{Remark}

\numberwithin{equation}{section} 
\numberwithin{figure}{section}
\numberwithin{table}{section}

\floatstyle{plaintop}
\newfloat{recipe}{thp}{lor}
\floatname{recipe}{Recipe}
\numberwithin{recipe}{section}

\providecommand{\mathbold}[1]{\bm{#1}}  %

\renewcommand{\phi}{\varphi}

\renewcommand{\mid}{\mathrel{\mathop{:}}}

\newcommand{\Id}{\mathbf{I}}

\providecommand{\mathbbm}{\mathbb} %

\newcommand{\IN}{\mathbbm{N}}

\newcommand{\polar}{\circ}

\newcommand{\argmin}{\operatorname*{arg\; min}}

\newcommand{\Prob}{\mathbbm{P}}

\newcommand{\Expect}{\operatorname{\mathbb{E}}}

\DeclareMathOperator{\Gl}{Gl}

\newcommand{\vct}[1]{\mathbold{#1}}

\newcommand{\IR}{\mathbbm{R}}

\newcommand{\mO}{\mathcal{O}}
\newcommand{\mL}{\mathcal{L}}
\newcommand{\mU}{\mathcal{U}}

\newcommand{\lin}{\operatorname{lin}}

\newcommand{\inter}{{\operatorname{int}}}

\newcommand{\Lk}{L^{(k)}}

\newcommand{\GL}{\operatorname{GL}}
\newcommand{\rev}{\operatorname{rev}}

\newcommand{\LEFT}{\textsc{Left}}
\newcommand{\RIGHT}{\textsc{Right}}

\newcommand{\Disjunionu}[1]{\underset{#1}{+\hspace{-3.5mm}\bigcup}}
\newcommand{\Disjunionuo}[2]{\overset{#2}{\underset{#1}{+\hspace{-3.5mm}\bigcup}}}

\newcommand{\Skel}{S}
\newcommand{\mSkel}{\mathcal{S}}

\DeclareMathOperator{\add}{add}

\newcommand{\vp}{\varphi}

\DeclareMathOperator{\spa}{span}

\DeclareMathOperator{\Lin}{Lin}
\DeclareMathOperator{\BL}{BL}
\newcommand{\BLs}{\BL\!\!^*}

\renewcommand{\P}{\mathcal{P}}

\newcommand{\invadj}{\circ}

\newcommand{\M}{\mathcal{M}}
\newcommand{\N}{\mathcal{N}}

\newcommand{\Q}{\mathcal{Q}}

\newcommand{\bprod}{\,\hat\times\,}
\newcommand{\Dir}{\Delta}
\newcommand{\mDir}{\hat{\Delta}}
\newcommand{\bproj}{\hat\Pi}

\newcommand{\mB}{\mathscr{B}}
\newcommand{\hmB}{\hat{\mathscr{B}}}
\newcommand{\mM}{\mathbbm{M}}
\newcommand{\hmM}{\hat{\mathbbm{M}}}

\title[Measures on polyhedral cones]{Measures on polyhedral cones: \\ characterizations and kinematic formulas}

\author[D.~Amelunxen]{Dennis Amelunxen}

\date{\today}

\begin{document}

\begin{abstract}
This paper is about conic intrinsic volumes and their associated integral geometry. We pay special attention to the biconic localizations of the conic intrinsic volumes, the so-called support measures. An analysis of these quantities has so far been confined to the PhD thesis of Stefan Glasauer (1995). We rederive the results from this thesis with novel streamlined proofs and expand them in several ways. Additionally, we introduce a new class of functionals on polyhedral cones lying between the intrinsic volumes and the well-studied $\vct f$-vector, which counts the equidimensional faces of a cone, and derive a characterization and kinematic formulas for these functionals as well.
\end{abstract}

\maketitle

\section{Introduction}\label{sec:intro}

Spherical integral geometry has so far led a niche existence, outshone by the classical theory of Euclidean integral geometry whose roots date back to the 18th century with Buffon's famous needle problem~\cite{KR:97}. Although these two settings are closely related, they are different in some decisive aspects, and a discussion of the spherical context as an exotic version of the Euclidean one almost certainly falls short of providing a full view on this independent theory. Furthermore, recent developments~\cite{vesp:92,DT:05,dota:09a,edge} have shown that spherical, or conic, integral geometry can be extremely useful in some applied areas of mathematics such as numerical optimization and compressed sensing. This paper serves as an elementary and mostly self-contained introduction to the theory of conic integral geometry with an emphasis on the (conic and biconic) localizations. The deepest results so far have been achieved by Glasauer~\cite{Gl,Gl:summ}. But, although parts of his work has been included in~\cite[Sec.~6.5]{SW:08}, the proofs of his results on \emph{support measures} can solely be found in his thesis. Additionally, he adopts the spherical viewpoint, which makes the theory of localizations incomplete and raises unnecessary hurdles for the proofs. By consistently adopting the conic viewpoint we obtain a complete theory of localizations with streamlined proofs, which we present in a self-contained manner. See Section~\ref{sec:contribs} below for a more detailed description of the specific contributions of this paper.

In the following sections we introduce the conic intrinsic volumes, describe their surprisingly numerous parallels to the $\vct f$-vector of a cone, and give a summary of the main results. To avoid introducing too much notation in the introduction, the statements of these results will be given in their nonlocalized versions while the more general forms are deferred to the main body of the paper.

\subsection{Counting and measuring faces}

A set $C\subseteq\IR^d$ is a \emph{(convex) polyhedral cone} if it can be described as the intersection of finitely many closed half-spaces, that is, $C=\{\vct x\in\IR^d\mid \vct{Ax}\leq\vct0\}$ for some matrix $\vct A\in\IR^{m\times d}$. Equivalently, $C\subseteq\IR^d$ is a polyhedral cone if it is the nonnegative linear hull of a finite set of vectors in $\IR^d$, $C=\{\vct{By}\mid \vct y\in\IR^k, \vct y\geq0\}$ for some matrix $\vct B\in\IR^{d\times k}$. We denote the set of polyhedral cones in~$\IR^d$ by~$\P(\IR^d)$.

A subset $F\subseteq C$ of a polyhedral cone~$C$ is a \emph{face} of~$C$ if there exists $\vct z\in\IR^d$ such that
\begin{equation}\label{eq:exp-face}
  \langle \vct x,\vct z\rangle \leq 0 \text{ for all } \vct x\in C \quad\text{and}\quad F = \{\vct x\in C\mid \langle\vct x,\vct z\rangle=0\} .
\end{equation}
Clearly, every face of a polyhedral cone is again a polyhedral cone, and choosing $\vct z=\vct 0$ shows that in particular~$C$ itself is a face of~$C$. A polyhedral cone has finitely many faces, and the \emph{$\vct f$-vector} counts these faces according to their dimensions: for $C\in\P(\IR^d)$,
\begin{align}\label{eq:f-vect}
   \vct f(C) & = \big(f_0(C),\ldots,f_d(C)\big) , & f_k(C) & := \big|\{F\mid F \text{ face of } C \text{ with } \dim(\spa(F))=k\}\big| ,
\end{align}
where $\spa(F) := F-F$ denotes the linear span of~$F$. Although the $\vct f$-vector will not be part of the investigations of this paper, we recall some of its fundamental properties~\cite{Z:95}:
\begin{enumerate}\setcounter{enumi}{-1}
  \item \emph{linear invariance:} The $\vct f$-vector is invariant under nondegenerate linear transformations, $\vct f(\vct TC)=\vct f(C)$ for $C\in\P(\IR^d)$ and $\vct T\in\Gl_d$.
  \item \emph{subspaces:} If $L\subseteq\IR^d$ is an $m$-dimensional linear subspace, then $f_m(L)=1$ and $f_k(L)=0$ for all $k\neq m$.
  \item \emph{products:} The $\vct f$-vector of the product $C\times D$ of polyhedral cones~$C,D$ is given by the convolution of the $\vct f$-vectors of~$C$ and~$D$,
  \begin{equation}\label{eq:f_k(CxD)=...}
    f_k(C\times D) = \sum_{i+j=k} f_i(C)f_j(D) .
  \end{equation}
  \item \emph{polarity:} Denoting the \emph{polar cone} of~$C\in\P(\IR^d)$ by $C^\polar := \{\vct z\in\IR^d\mid \langle \vct x,\vct z\rangle\leq 0 \text{ for all } \vct x\in C \} \in\P(\IR^d)$, the $\vct f$-vector of $C^\polar$ is the reverse $\vct f$-vector of~$C$,
  \begin{equation}\label{eq:f_k(C^polar)}
    f_k(C^\polar) = f_{d-k}(C) .
  \end{equation}
  \item \emph{Euler property:} The Euler characteristic yields for every polyhedral cone $C\in\P(\IR^d)$, which is not a linear subspace,
  \begin{equation}\label{eq:Euler-f}
    \sum_{k=0}^d (-1)^k f_k(C) = 0 .
  \end{equation}
\end{enumerate}

This paper is not about the $\vct f$-vector\footnote{See~\cite{S:14} for a recent survey on this impressive theory.} but about certain variants of it, which may be conceived as ``weighted versions'' of the $\vct f$-vector. For this we introduce the following notation: the set of linear spans of the ($k$-dimensional) faces of $C\in\P(\IR^d)$ shall be denoted by
  \[ \mL(C) := \{\spa(F)\mid F\text{ face of } C\} ,\qquad \mL_k(C) := \{L\in\mL(C)\mid \dim(L)=k\} . \]
Note that $L\cap C$ with $L\in\mL(C)$ is a face of~$C$, and all faces of~$C$ are of this form. In particular, $\mL(C)$ is a finite set and $f_k(C)=|\mL_k(C)|$. The supporting planes of a cone~$C\in\P(\IR^d)$ and of its polar~$C^\polar$ are related via the bijections
\begin{equation}\label{eq:L_k(C)--L_(d-k)(C^polar)}
  \mL_k(C)\to\mL_{d-k}(C^\polar) ,\qquad L\mapsto L^\bot ,
\end{equation}
where $L^\bot:=\{\vct z\in\IR^d\mid \langle \vct x,\vct z\rangle=0 \text{ for all } \vct x\in L\} \, (=L^\polar)$ denotes the orthogonal complement of~$L$. 

One of the most important, yet entirely simple, results, which we will make heavy use of in this paper, is that a polyhedral cone decomposes disjointly into the relative interiors of its faces, i.e.,
\begin{align}\label{eq:fac-decomp-C}
  C & = \Disjunionu{L\in\mL(C)} \inter_L(C\cap L) = \Disjunionuo{k=0}{d} \Skel_k(C) , & \Skel_k(C) & := \Disjunionu{L\in\mL_k(C)} \inter_L(C\cap L) ,
\end{align}
where $\inter_L$ shall denote the interior with respect to the relative topology in~$L$. The set $\Skel_k(C)$ is called the \emph{$k$-skeleton} of~$C$. Furthermore, we denote the standard Gaussian measure on~$L$ by $\gamma_L$ and abbreviate $\gamma_d:=\gamma_{\IR^d}$. In the case $\dim L=0$ this coincides with the Dirac measure supported at the origin, for which we use the notation $\Dir=\gamma_0$.

The \emph{$\vct u$-vector} of~$C\in\P(\IR^d)$ collects the sums of the Gaussian volumes of equidimensional faces of~$C$:
\begin{align}\label{eq:def-u-vect}
   \vct u(C) & = \big(u_0(C),\ldots,u_d(C)\big) , & u_k(C) & := \sum_{L\in\mL_k(C)} \gamma_L(C\cap L) .
\end{align}
Note that $u_0(C)=f_0(C)$ and $u_d(C)=\gamma_d(C)$. The $\vct u$-vector is a weaker invariant than the $\vct f$-vector; clearly, linear transformations can change the Gaussian volumes of the faces. However, we still have orthogonal invariance:
\begin{itemize}
  \item[(0$^*$)] orthogonal invariance: The $\vct u$-vector is invariant under orthogonal transformations, $\vct u(\vct QC)=\vct u(C)$ for $C\in\P(\IR^d)$ and $\vct Q\in O(d)$.
\end{itemize}
The $\vct u$-vector satisfies the inequalities $0\leq u_k\leq f_k$, and for linear subspaces the $\vct u$- and the $\vct f$-vector coincide, so that the $\vct u$-vector also has property~(1). Furthermore, the $\vct u$-vector also satisfies the product rule~(2), cf.~\eqref{eq:f_k(CxD)=...}.

On the other hand, the $\vct u$-vector possesses neither a polarity property as in~\eqref{eq:f_k(C^polar)} nor an Euler property as in~\eqref{eq:Euler-f}. The $\vct v$-vector, to be introduced next, may be thought of as the straightforward way to reestablish the polarity relation, which, almost mysteriously, not only reestablishes the Euler property as well, but also yields further fundamental properties.

The \emph{$\vct v$-vector} of~$C\in\P(\IR^d)$ collects the \emph{(conic) intrinsic volumes}~\cite{McM:75} of~$C$:
\begin{align}\label{eq:def-v-vect}
   \vct v(C) & = \big(v_0(C),\ldots,v_d(C)\big) , & v_k(C) & := \sum_{L\in\mL_k(C)} \gamma_L(C\cap L) \; \gamma_{L^\bot}(C^\polar\cap L^\bot) .
\end{align}
Alternatively, one can define the intrinsic volume of a polyhedral cone by combining the decomposition~\eqref{eq:fac-decomp-C} with the projection map $\Pi_C\colon\IR^d\to C$, $\Pi_C(\vct x)=\argmin\{\|\vct x-\vct y\|\mid \vct y\in C\}$: denoting by $\vct g\in\IR^d$ a standard Gaussian vector, we have
\begin{equation}\label{eq:def-v-vect-Prob}
   v_k(C) = \Prob\big\{\Pi_C(\vct g)\in\Skel_k(C)\big\} .
\end{equation}

From the definition~\eqref{eq:def-v-vect} it is immediate that $0\leq v_k\leq u_k$. Note also that from the characterization~\eqref{eq:def-v-vect-Prob} one directly obtains $v_0(C)+\cdots+v_d(C) = 1$. So the $\vct v$-vector is actually a discrete probability distribution.\footnote{A loose but intuitive interpretation of $v_k(C)$ is that it describes the ``amount of $k$-dimensionality'' of $C$; if $v_7(C)=0.23$ then $C$ is $23\%$ $7$-dimensional.} Furthermore, the $\vct v$-vector is orthogonal invariant, i.e., it satisfies property~($0^*$), and it satisfies the subspace and the product rules~(1) and~(2), just as the $\vct u$-vector does. Unlike the $\vct u$-vector, the $\vct v$-vector also satisfies the polarity property~(3), which follows from the bijection~\eqref{eq:L_k(C)--L_(d-k)(C^polar)} between the supporting subspaces of~$C$ and~$C^\polar$. Less obvious is the fact that the $\vct v$-vector even satisfies the Euler property~(4). But more than that, the $\vct v$-vector possesses two additional remarkable properties, which are not shared by the $\vct f$-vector:
\begin{enumerate}\setcounter{enumi}{4}
  \item \emph{additivity:} If $C,D\in\P(\IR^d)$ are such that $C+D=C\cup D$, or equivalently, $C\cup D\in\P(\IR^d)$, then
  \begin{equation}\label{eq:additivity-v}
    \vct v(C \cup D) + \vct v(C\cap D) = \vct v(C) + \vct v(D) .
  \end{equation}
  \item \emph{continuity:} If $C_i\in\P(\IR^d)$, $i\in\IN$, such that $\lim_{i\to\infty} C_i=C\in\P(\IR^d)$ with respect to the conic Hausdorff metric\footnote{The conic Hausdorff metric is just the spherical Hausdorff metric which is obtained by replacing a cone by its intersection with the unit sphere, cf.~\cite[Sec.~3.2]{am:thesis}.}, then
  \begin{equation}\label{eq:continuity-v}
    \lim_{i\to\infty} \vct v(C_i)=\vct v(C) .
  \end{equation}
\end{enumerate}

\begin{remark}
The weaker invariance of the $\vct u$- and $\vct v$-vector may be bemoaned, but in fact this weaker invariance is a necessary requirement for the features, which we will discuss next, the kinematic formulas. These formulas treat, for example, the $\vct u$- or $\vct v$-vector of the (random) intersection of one cone with a randomly rotated second cone. It turns out that the \emph{expectation} of these random vectors can again be expressed in terms of the $\vct u$- and $\vct v$-vectors, respectively, of the components. Formulas of this kind do not hold for the $\vct f$-vector simply because of its linear invariance: taking linear transformations of the components does not change their $\vct f$-vectors, but it does of course change the probabilities for the random intersections. So in a sense we have traded the strong invariance of the $\vct f$-vector for new probabilistic formulas. Strangely enough, we will see that through the kinematic formulas we will also regain linear invariance \emph{in expectation},~cf.~\eqref{eq:TQC-u,v} below.
\end{remark}

See Table~\ref{tab:props-f-u-v} for an overview of the properties of the $\vct f$-/$\vct u$-/$\vct v$-vectors.

\begin{table}
  \begin{center}
\begin{tabu}{|c|c|c|}
   \hline $\vct f$ & $\vct u$ & $\vct v$
\\ \tabucline[1pt]{-}
   $f_k\in\IN$ & $0\leq u_k\leq f_k$ & $0\leq v_k\leq u_k$, $\sum_k v_k = 1$
\\\hline $\vct f(\vct TC)=\vct f(C)$ &
$\begin{array}{r@{\,}c@{\,}l} \vct u(\vct QC) & = & \vct u(C) \\ \Expect\big[\vct u(\vct{TQ}C)\big] & = & \vct u(C) \end{array}$
  &
$\begin{array}{r@{\,}c@{\,}l} \vct v(\vct QC) & = & \vct v(C) \\ \Expect\big[\vct v(\vct{TQ}C)\big] & = & \vct v(C) \end{array}$
\\\hline $f_k(L) = \begin{cases} 1 & \text{if } k=\dim L \\ 0 & \text{else} \end{cases}$ & $u_k(L) = \begin{cases} 1 & \text{if } k=\dim L \\ 0 & \text{else} \end{cases}$ & $v_k(L) = \begin{cases} 1 & \text{if } k=\dim L \\ 0 & \text{else} \end{cases}$
\\\hline $f_k(C) = f_{d-k}(C^\polar)$ & -- & $v_k(C) = v_{d-k}(C^\polar)$
\\\hline $\displaystyle f_k(C\times D) = \sum_{i+j=k} f_i(C)f_j(D)$ & $\displaystyle u_k(C\times D) = \sum_{i+j=k} u_i(C)u_j(D)$ & \rule{0mm}{4mm} $\displaystyle v_k(C\times D) = \sum_{i+j=k} v_i(C)v_j(D)$
\\\hline $\begin{matrix} \sum_k (-1)^k f_k(C)=0 \\ \text{if $C$ not a linear subspace}\end{matrix}$ & -- & $\begin{matrix} \sum_k (-1)^k v_k(C)=0 \\ \text{if $C$ not a linear subspace}\end{matrix}$
\\\hline --
  & $\begin{matrix} \Expect\big[ u_k(C\cap\vct QD)\big] = \rule{18mm}{0mm}\rule{0mm}{4mm} \\[2mm] \displaystyle\sum_{i+j=d+k} u_i(C)u_j(D) \quad\text{ if } k>0 \end{matrix}$
  & $\begin{matrix} \Expect\big[ v_k(C\cap\vct QD)\big] = \rule{18mm}{0mm}\rule{0mm}{4mm} \\[2mm] \displaystyle\sum_{i+j=d+k} v_i(C)v_j(D) \quad\text{ if } k>0 \end{matrix}$
\\\hline -- & -- & $\begin{matrix} \vct v(C\cup D)+\vct v(C\cap D)=\vct v(C)+\vct v(D) \\ \text{if } C\cup D\in\P(\IR^d) \end{matrix}$
\\\hline -- & -- & $\displaystyle \lim_{i\to\infty} C_i=C \Rightarrow \lim_{i\to\infty} \vct v(C_i)=\vct v(C)$
\\\hline
\end{tabu}
  \end{center}
  \caption{Elementary properties of the $\vct f,\vct u,\vct v$-vectors: $C,D,C_i\in\P(\IR^d)$, $\vct T\in\Gl_d$, $\vct Q\in O(d)$, $L\subseteq\IR^d$ linear subspace, the expectations are with respect to $\vct Q\in O(d)$ uniformly at random.}
  \label{tab:props-f-u-v}
\end{table}

\subsection{Conic kinematic formulas}

Before stating the kinematic formulas, note that iteratively applying the product rule~(2), cf.~\eqref{eq:f_k(CxD)=...}, yields for $C_0,\ldots,C_n\in\P(\IR^d)$,
  \[ u_m(C_0\times\cdots\times C_n) = \sum_{i_0+\cdots+i_n=m} u_{i_0}(C_0)\cdots u_{i_n}(C_n) , \]
and similarly for $v_m(C_0\times\cdots\times C_n)$. In the following we say that $\vct Q\in O(d)$ uniformly at random if~$\vct Q$ is distributed according to the normalized Haar measure; $\Expect[\cdots]$ shall denote the expectation.

\begin{theorem}[Kinematic formulas for $\vct u$ and $\vct v$]\label{thm:kinform-u,v}
Let $C_0,\ldots,C_n\in\P(\IR^d)$ and $\vct T_0,\ldots,\vct T_n\in\Gl_d$, and let $k>0$. Then for $\vct Q_0,\ldots,\vct Q_n\in O(d)$ iid uniformly at random,
\begin{align}
   \Expect\bigg[ u_k\bigg( \bigcap_{i=0}^n \vct T_i\vct Q_i C_i \bigg)\bigg] & = u_{nd+k}(C_0\times\cdots\times C_n) ,
\label{eq:kinform-u}
\\ \Expect\bigg[ v_k\bigg( \bigcap_{i=0}^n \vct T_i\vct Q_i C_i \bigg)\bigg] & = v_{nd+k}(C_0\times\cdots\times C_n) .
\label{eq:kinform-v}
\end{align}
\end{theorem}

Note that as special cases of~\eqref{eq:kinform-u} and~\eqref{eq:kinform-v}, we obtain for $C\in\P(\IR^d)$, $\vct T\in\GL_d$, and $\vct Q\in O(d)$ uniformly at random,
\begin{align}\label{eq:TQC-u,v}
  \Expect\big[\vct u(\vct{TQ}C)\big] & = \vct u(C) , & \Expect\big[\vct v(\vct{TQ}C)\big] & = \vct v(C) .
\end{align}
So although $\vct u$ and $\vct v$ in general both fail to be linear invariants, they are still linear invariants \emph{in expectation}.

The additional polarity property of the $\vct v$-vector has important consequences for the kinematics. For example, since $(C\cap D)^\polar=C^\polar+D^\polar$ we immediately obtain the following corollary from Theorem~\ref{thm:kinform-u,v} (see Section~\ref{sec:prelims-polyh-cones} for some subtleties involving the linear transformations $\vct T_0,\ldots,\vct T_n$).

\begin{corollary}[Polar kinematic formula for $\vct v$]\label{cor:kinform-v-polar}
Let the notation and assumptions be as in Theorem~\ref{thm:kinform-u,v}. Then
\begin{align}\label{eq:kinform-v-polar}
   \Expect\bigg[ v_{d-k}\bigg( \sum_{i=0}^n \vct T_i\vct Q_i C_i \bigg)\bigg] & = v_{d-k}(C_0\times\cdots\times C_n) .
\end{align}
\end{corollary}

Moreover, the fact that $\vct v$ is a probability distribution, i.e., $v_0+\cdots+v_d=1$, and linearity of expectation yield the following formulas for the boundary cases in~\eqref{eq:kinform-v} and~\eqref{eq:kinform-v-polar}.

\begin{corollary}[Boundary cases for $\vct v$]\label{cor:kinform-v-bd_cases}
Let the notation be as in Theorem~\ref{thm:kinform-u,v}. Then
\begin{align}\label{eq:kinform-v-bd_cases}
   \Expect\bigg[ v_0\bigg( \bigcap_{i=0}^n \vct T_i\vct Q_i C_i \bigg)\bigg] & = \sum_{j=0}^{nd} v_j(C_0\times\cdots\times C_n) , & \Expect\bigg[ v_d\bigg( \sum_{i=0}^n \vct T_i\vct Q_i C_i \bigg)\bigg] & = \sum_{j=0}^{nd} v_{d+j}(C_0\times\cdots\times C_n) .
\end{align}
\end{corollary}

\begin{remark}
If one combines the kinematic formulas for the conic intrinsic volumes with the Euler property $\sum_k (-1)^k v_k(C)=0$ if $C$ not a linear subspace, then one obtains the so-called \emph{Crofton formulas}, which describe the intersection probabilities of randomly rotated cones. These Crofton formulas form the link between the theory of conic integral geometry and the applications in optimization and compressed sensing. See~\cite{edge,MT:13} for further details.
\end{remark}

The different kinematic formulas for~$\vct v$ given in~\eqref{eq:kinform-v}, \eqref{eq:kinform-v-polar},~\eqref{eq:kinform-v-bd_cases} can in fact be seen as special cases\footnote{We only prove the general kinematic formula~\eqref{eq:gen-kinform-v} with the restriction $\vct T_0,\ldots,\vct T_n\in O(d)$, cf.~Remark~\ref{rem:gen-kinform}.} of a general kinematic formula, which we formulate next.

If $F(X_0,\ldots,X_n)$ denotes a Boolean formula in the variables $X_0,\ldots,X_n$, and if $C_0,\ldots,C_n\in\P(\IR^d)$, we define the evaluation $F(C_0,\ldots,C_n)\in\P(\IR^d)$ to be the result of the following replacements in the formula $F(X_0,\ldots,X_n)$:
  \[\begin{tabu}{r|[1pt]c|c|c|c}
       \text{replace} & X_i & \wedge & \vee & \neg(\cdots)
    \\\hline   
       \text{by} & C_i & \cap & + & (\cdots)^\polar
    \end{tabu} \quad . \]
So if, for example $F(X_0,X_1,X_2) = \neg(X_0\wedge X_1)\vee X_2$, then $F(C_0,C_1,C_2)=(C_0\cap C_1)^\polar + C_2$. Note that $C+D = C\cup D$ if the latter is a convex cone, which leads to the more familiar case where the logical~$\vee$ corresponds to~$\cup$.

A Boolean formula in which every variable appears at most once is called a \emph{Boolean read-once formula}. If $F(X_0,\ldots,X_n)$ is a Boolean read-once formula and if $L_0,\ldots,L_n\subseteq\IR^d$ are linear subspaces in general position,\footnote{See Appendix~\ref{sec:genericity} for more details on these genericity assumptions.} then $F(L_0,\ldots,L_n)$ is again a linear subspace in general position whose dimension only depends on~$F$, the ambient dimension~$d$, and the dimensions of $L_0,\ldots,L_n$. We may thus define for every read-once formula~$F(X_0,\ldots,X_n)$ and for any dimension~$d$ the function
\begin{align}
   \dim^F_d & \colon \{0,\ldots,d\}^{n+1} \to \{0,\ldots,d\} , & \dim^F_d(k_0,\ldots,k_n) & = \dim\big( F(L_0,\ldots,L_n) \big) ,
\end{align}
where $L_0,\ldots,L_n\subseteq\IR^d$ are linear subspaces in general position with $\dim(L_i)=k_i$. For example, in the cases $F(X_0,\ldots,X_n)=X_0\wedge\cdots\wedge X_n$ and $F(X_0,\ldots,X_n)=X_0\vee\cdots\vee X_n$ we obtain
  \[\begin{tabu}{r|[1pt]c|c}
       F(X_0,\ldots,X_n) = & X_0\wedge\cdots\wedge X_n & X_0\vee\cdots\vee X_n
    \\\hline
       \dim^F_d(k_0,\ldots,k_n) = & \max\big\{ 0, k_0+\cdots+k_n-nd \big\} & \min\big\{ d, k_0+\cdots+k_n \big\}
    \end{tabu} \quad . \]

In the following theorem we formulate a generalization of the kinematic formulas for~$\vct v$ given in~\eqref{eq:kinform-v}, \eqref{eq:kinform-v-polar}, and~\eqref{eq:kinform-v-bd_cases}, which we prove in the special case $\vct T_0,\ldots,\vct T_n\in O(d)$. Of course, if $\vct T_i\in O(d)$ fixed and $\vct Q_i\in O(d)$ uniformly at random, then $\vct T_i\vct Q_i\in O(d)$ uniformly at random, so we might as well drop the $\vct T_i$ here. But we include them nevertheless to ease the comparison with the other formulas and also because the restriction $\vct T_i\vct Q_i\in O(d)$ should be seen as an (unsubstantial) artefact of the proof, cf.~also Remark~\ref{rem:gen-kinform}.

\begin{theorem}[General kinematic formula for $\vct v$]\label{thm:gen-kinform-v}
Let $C_0,\ldots,C_n\in\P(\IR^d)$ and $\vct T_0,\ldots,\vct T_n\in O(d)$, and let $0\leq k\leq d$. Furthermore, let $F(X_0,\ldots,X_n)$ be a Boolean read-once formula. Then for $\vct Q_0,\ldots,\vct Q_n\in O(d)$ iid uniformly at random,
\begin{equation}\label{eq:gen-kinform-v}
  \Expect\bigg[ v_k\Big( F\big(\vct T_0\vct Q_0 C_0,\ldots,\vct T_n\vct Q_n C_n\big) \Big)\bigg] = \sum_{\dim^F_d(k_0,\ldots,k_n)=k} v_{k_0}(C_0) \cdots v_{k_n}(C_n) .
\end{equation}
\end{theorem}

Note that~\eqref{eq:kinform-v} is~\eqref{eq:gen-kinform-v} in the special case $F(X_0,\ldots,X_n)=X_0\wedge\cdots\wedge X_n$, while~\eqref{eq:kinform-v-polar} is the case $F(X_0,\ldots,X_n)=X_0\vee\cdots\vee X_n$. Likewise, the boundary cases in~\eqref{eq:kinform-v-bd_cases} are covered by these choices for~$F$.

\begin{remark}\label{rem:gen-kinform}
Only in dimension $d=1$ the general kinematic formula~\eqref{eq:gen-kinform-v} is true for \emph{every} Boolean formula; for $d\geq2$ one can find counterexamples, like the formula $F(X_0,X_1)=(X_0\vee X_1)\wedge \neg X_0$, for which the general kinematic formula~\eqref{eq:gen-kinform-v} \emph{fails}.

Again, since $\vct T_0,\ldots,\vct T_n$ are assumed to be orthogonal, their inclusion in~\eqref{eq:gen-kinform-v} does not constitute a generalization of the case $\vct T_0=\cdots=\vct T_n=\Id_d$. But we choose to write them nevertheless also to emphasize that the inclusion of general (nondegenerate) transformations does affect the proof strategy for this kind of generalizations of the kinematic formula. In the proof given in Section~\ref{sec:gen-kin-form} we will say explicitly where the orthogonality assumption is used; the same proof strategy does not seem to yield the general case $\vct T_0,\ldots,\vct T_n\in \Gl_d$.
\end{remark}

\subsection{Localizations}\label{sec:localizations}

The key to proving the kinematic formulas provided in the previous section is to consider \emph{localizations} of~$\vct u$ and~$\vct v$. In this section we give the definitions of these localizations and also provide a description of the general proof strategy for the kinematic formulas. For details about the involved characterization theorems we refer to the following sections.

We say that a Borel set $M\subseteq\IR^d$ is a \emph{conic Borel set} if it is invariant under positive scaling, i.e., $\lambda M=M$ for all $\lambda>0$. The set of conic Borel sets is called the \emph{conic (Borel) $\sigma$-algebra} on~$\IR^d$, denoted
\begin{align}\label{eq:def-B(R^d)}
   \hmB(\IR^d) & := \{ M\subseteq\IR^d \text{ Borel set}\mid \lambda M=M \text{ for all } \lambda>0 \} .
\end{align}
We denote the set of $\sigma$-additive real functions defined on this algebra by~$\hmM(\IR^d)$, the set of \emph{conic measures}.
The most important conic measures are the Dirac measure~$\Dir$ supported at the origin and the (standard) Gaussian measure~$\gamma_d$. Every orthogonal invariant conic measure is a linear combination of~$\Dir$ and~$\gamma_d$, cf.~Lemma~\ref{lem:orth-inv-conic-meas}.

To get the connection to the $\vct u$- and $\vct v$-vectors, one considers families of conic measures, which are parametrized by polyhedral cones. The \emph{polyhedral measures}~$\Psi_k(C,\cdot)$ and the \emph{curvature measures}~$\Phi_k(C,\cdot)$, where $C\in\P(\IR^d)$ and $0\leq k\leq d$, are defined by
\begin{align}
   \Psi_k(C,M) & := \sum_{L\in\mL_k(C)} \gamma_L(C\cap L\cap M) ,
\label{eq:def-Psi_k}
\\ \Phi_k(C,M) & := \sum_{L\in\mL_k(C)} \gamma_L(C\cap L\cap M) \, \gamma_{L^\bot}(C^\polar\cap L^\bot) ,
\label{eq:def-Phi_k}
\end{align}
where $M\in\hmB(\IR^d)$. Note that $\Psi_k(C,\IR^d)=u_k$ and $\Phi_k(C,\IR^d)=v_k$, so the polyhedral measures and the curvature measures are localizations of the $\vct u$- and $\vct v$-vector, respectively. An alternative characterization for the curvature measures is given by
\begin{equation}\label{eq:def-Phi-vect-Prob}
   \Phi_k(C,M) = \Prob\big\{\Pi_C(\vct g)\in\Skel_k(C)\cap M\big\} ,
\end{equation}
where $\vct g\in\IR^d$ denotes a standard Gaussian vector, cf.~\eqref{eq:def-v-vect-Prob}.

The general proof strategy for the kinematic formula of the $\vct u$-vector~\eqref{eq:kinform-u} consists of proving a characterization theorem for the polyhedral measures (see Theorem~\ref{thm:conic-char}), and to use this to prove a kinematic formula for the polyhedral measures (see Theorem~\ref{thm:kinform-polyh}), which specializes to~\eqref{eq:kinform-u}. As for the intrinsic volumes~$\vct v$ we could proceed similarly with the curvature measures. But a characterization of these measures is a bit more complicated, which makes the curvature measures a less favorable tool. So instead, we will use some other measures which require introducing a bit more notation, but which share much of the simplicity and naturalness of the polyhedral measures.

The \emph{biconic (Borel) $\sigma$-algebra} on~$\IR^d\times\IR^d=:\IR^{d+d}$ is defined by
\begin{align}\label{eq:def-B(R^d,R^d)}
  \hmB(\IR^d,\IR^d) & := \{ \M\subseteq\IR^{d+d} \text{ Borel set}\mid (\lambda,\lambda') \M=\M \text{ for all }\lambda,\lambda'>0 \} ,
\end{align}
where $(\lambda,\lambda')\M:=\{(\lambda\vct x,\lambda'\vct x')\mid (\vct x,\vct x')\in\M\}$. Again, we consider families of biconic measures parametrized by polyhedral cones. The \emph{support measures}~$\Theta_k(C,\cdot)$, where $C\in\P(\IR^d)$ and $0\leq k\leq d$, are defined by
\begin{align}\label{eq:def-Theta_k}
   \Theta_k(C,\M) & := \Prob\Big\{\Pi_C(\vct g)\in\Skel_k(C) \text{ and } \big(\Pi_C(\vct g),\Pi_{C^\polar}(\vct g)\big)\in\M \Big\} ,
\end{align}
where $\Pi_C\colon\IR^d\to C$ denotes again the orthogonal projection map, and where $\vct g\in\IR^d$ is a standard Gaussian vector. Comparing this with~\eqref{eq:def-v-vect-Prob} we see that $v_k(C)=\Theta_k(C,\IR^{d+d})$, so also the support measures are localizations (biconic this time) of the intrinsic volumes. Moreover, if the biconic set is a direct product $\M=M\times M'$, then we obtain
\begin{equation}\label{eq:def-Theta_k-dirprod}
  \Theta_k(C,M\times M') = \sum_{L\in\mL_k(C)} \gamma_L(C\cap L\cap M) \, \gamma_{L^\bot}(C^\polar\cap L^\bot\cap M') .
\end{equation}
Therefore, $\Theta_k(C,M\times\IR^d)=\Phi_k(C,M)$, so the support measures also localize the curvature measures. In fact, the support measures, due to their inherent symmetry between the primal cone~$C$ and its polar~$C^\polar$, appear to be the more natural choice for a localization of the intrinsic volumes than the curvature measures. This impression is further supported by the specific form of the characterization theorem for the support measures, which shares much of the simplicity with that of the polyhedral measures; its proof is almost as elementary. We will exploit this characterization in a similar way as for the polyhedral measures and derive a corresponding (biconically) localized version of~\eqref{eq:kinform-v} in Theorem~\ref{thm:kinform-supp}.

\subsection{Outline}\label{sec:outline}

The organization of the paper is as follows. In Section~\ref{sec:conic-localizations} we present the theory around the \emph{conic} localizations. After two short preliminary sections we discuss the polyhedral and curvature measures in Section~\ref{sec:polyh-curv-meas}, and we prove a characterization theorem for the polyhedral measures in Section~\ref{sec:char-polyh}. From this we derive the kinematic formula for the polyhedral measures in Section~\ref{sec:calc-kinem-conic}, which generalizes the kinematic formula~\eqref{eq:kinform-u} for the $\vct u$-vector. Section~\ref{sec:biconic-localizations} follows the same structure, but in the \emph{biconic} setting: After a preliminary section on the biconic $\sigma$-algebra we discuss in Section~\ref{sec:suppmeas} the support measures, for which we prove a characterization theorem in Section~\ref{sec:char-supp-meas}, which is then used to prove the kinematic formula for the support measures in Section~\ref{sec:kinform-suppmeas}. In Section~\ref{sec:gen-kin-form} we prove the extension of the kinematic formula, as given in Theorem~\ref{thm:gen-kinform-v}. We also give an analogous generalized kinematic formula for the support measures, which can be proved in the same way.

We supplement this with two appendices. The first of these is devoted to the Steiner formulas, which play an important role in generalizing the theory around the intrinsic volumes, the curvature measures, and the support measures from polyhedral to general convex cones. Since we limit the discussion in this paper to polyhedral cones, we will not make extensive use of these formulas, but we include a short discussion for completeness. Appendix~\ref{sec:genericity} is devoted to settling some subtleties that arise in the proofs of the kinematic formulas. We chose to move these technicalities to the appendix to improve the readability of the proofs of the kinematic formulas in Sections~\ref{sec:calc-kinem-conic} and~\ref{sec:kinform-suppmeas}. These genericity statements are also intuitively clear so that a relocation to the appendix seems natural.

\subsection{Contributions}\label{sec:contribs}

Integral geometry in spaces of constant curvature goes back to at least the work of Santal\'o~\cite{Sant}; see also the work by Howard~\cite{H:93}, and~\cite{AF:14} for more recent developments in the theory of valuations on manifolds. This paper does not follow this line of research, which mostly uses differential geometric methods and does not use the specific features of the conic setting. Instead, we present an elementary approach to the theory of support measures of convex \emph{polyhedral} cones exploiting the specific properties of the conic context. An entirely new aspect is here that the classical approach of first proving a characterization theorem and then deriving from this certain kinematic formulas, does not only apply to the localizations of the intrinsic volumes, the curvature/support measures, but also applies to a new class of functionals, the polyhedral measures. This new aspect also illustrates the assessment that only the \emph{biconic} viewpoint naturally leads to the support measures (and thus to the curvature measures and intrinsic volumes) while the \emph{conic} viewpoint naturally leads to the polyhedral measures, which do not satisfy an additivity property, and thus are not part of the existing literature on conic integral geometry.

As already mentioned right from the start, the bulk of this paper is based on material from Glasauer's thesis~\cite{Gl,Gl:summ}. However, we have carried out a number of changes and additions to justify a separate paper. The most apparent change is that we do not use the spherical setting but argue in a conical context using the concepts of the conic and biconic $\sigma$-algebras, which were introduced in~\cite{ambu:12}. This seemingly superficial change has some important consequences. First of all, the conic versions of the curvature measures and the support measures also cover the boundary cases, which are excluded in the spherical theory. So the conic viewpoint allows a more complete theory. Furthermore, the proofs allow significant simplifications and streamlining, which makes the conic theory more accessible. 

The $\vct u$-vector and its localization given by the polyhedral measures are new concepts, which are apparently introduced here for the first time; the corresponding characterization and kinematic formulas are to that effect also new. These new measures are arguably more natural conic measures than the curvature measures. Indeed, the proof of the characterization theorem for the polyhedral measures (Section~\ref{sec:char-polyh}) in the end only consists of combining the elementary fact that the usual (Lebesgue) volume on the sphere is characterized through orthogonal invariance, with the facial decomposition of polyhedral cones. The characterization theorem for the curvature measures by Schneider, see for example~\cite{SW:08}, is more complicated and has to use the additivity property of the curvature measures (which, of course, the polyhedral measures do not possess). The proof of the characterization of the support measures (Section~\ref{sec:char-supp-meas}) follows the same line of arguments as that of the polyhedral measures and thus shares the same simplicity. The only difference is here that the biconic setting requires a bit more notation so that the proof may look more complicated on a first reading. The resulting kinematic formulas for the support measures (Section~\ref{sec:kinform-suppmeas}) are thus not only more general than those for the curvature measures, but in the end also have a (conceptually) simpler proof.

The duality property of the support measures, which is characteristic to the conic setting, is also needed to derive the localized projection formulas, cf.~Remark~\ref{rem:proj-form-curvmeas}. These formulas can be useful in the probabilistic analysis of convex programming, cf.~\cite[Rem.~2.1]{ambu:12}.

The equality $\Expect[\vct v(\vct{TQ}C)] = \vct v(C)$ in~\eqref{eq:TQC-u,v} has first been observed by Mike McCoy and Joel Tropp~\cite{MTpc:13}. Incorporating this newly discovered invariance (in expectation) into the kinematic formulas requires a new proof strategy in Sections~\ref{sec:calc-kinem-conic}/\ref{sec:kinform-suppmeas} than typically used, cf.~for example~\cite[Sec.~5.3.1]{McC:thesis}. This subtle point is made precise in the proof of Theorem~\ref{thm:gen-kinform-v} provided in Section~\ref{sec:gen-kin-form}.

\subsection{Acknowledgments}

Part of this research was carried out while spending a year at the mathematics department of The University of Manchester. Special thanks go to my collaborator and host in Manchester Martin Lotz for his broad support and for many useful and interesting discussions. I would also like to thank Mike McCoy and Joel Tropp as well as Peter Bürgisser for useful discussions on conic intrinsic volumes. This research was partly supported by DFG grant AM 386/1-2 and EPSRC grant EP/I01912X/1-CF05.

\section{Conic localizations}\label{sec:conic-localizations}

In this section we first provide in Sections~\ref{sec:prelims-polyh-cones} and~\ref{sec:prelim-conmeas} some preliminaries about polyhedral cones and about the conic $\sigma$-algebra, respectively, before we discuss the properties of the polyhedral measures and the curvature measures in Section~\ref{sec:polyh-curv-meas}. Section~\ref{sec:char-polyh} is devoted to the characterization of the polyhedral measures, which is then used in Section~\ref{sec:calc-kinem-conic} to prove a kinematic formula, which generalizes the kinematic formula for the $\vct u$-vector,~\eqref{eq:kinform-u}.

\subsection{Preliminaries: polyhedral cones}\label{sec:prelims-polyh-cones}

Most of the relevant notation and properties of the polyhedral cones, which we will use in this paper, has already been introduced in Section~\ref{sec:intro}. Here we supplement some further aspects.

The first of these aspects concerns (nonsingular) linear transformations~$\vct T\in\Gl_d$. The $\vct f$-vector is invariant under these transformations, $\vct f(\vct TC)=\vct f(C)$. More precisely, the faces of $\vct TC$ are all of the form $\vct TF$ for some face $F$ of $C$, or in terms of the corresponding linear subspaces,
  \[ \mL_k(\vct TC)=\big\{ \vct TL\mid L\in \mL_k(C)\big\} , \]
for $k=0,\ldots,d$. Recall that Corollary~\ref{cor:kinform-v-polar} is supposed to follow immediately from Theorem~\ref{thm:kinform-u,v} via polarity. More precisely, to deduce this one needs to know that the polar of~$\vct TC$ is given by
\begin{align}\label{eq:def-T^invadj}
   (\vct TC)^\polar & = \vct T^\invadj C^\polar ,  \quad\text{where} \qquad \vct T^\invadj := \big(\vct T^{-1}\big)^T = \big(\vct T^T\big)^{-1} .
\end{align}
Note that we have for $\vct S,\vct T\in \Gl_d$, $\vct Q\in O(d)$,
\begin{align*}
   \big(\vct T^\invadj\big)^\invadj & = \vct T , & (\vct S\vct T)^\invadj & = \vct S^\invadj\vct T^\invadj , & \vct Q^\invadj = \vct Q .
\end{align*}
Replacing $C_0,\ldots,C_n$ and $\vct T_0,\ldots,\vct T_n$ in Theorem~\ref{thm:kinform-u,v} by $C_0^\polar,\ldots,C_n^\polar$ and $\vct T_0^\invadj,\ldots,\vct T_n^\invadj$, respectively, and applying polarity yields Corollary~\ref{cor:kinform-v-polar}.

The next important aspect we need to address is that of direct products. The faces of $C\times D$, where $C,D\in\P(\IR^d)$, are given by the products of the faces of $C$ and $D$. In terms of the corresponding linear subspaces, we obtain
\begin{equation}\label{eq:L_k(CxD)=...}
  \mL_k(C\times D) = \Disjunionu{i+j=k} \big\{ L_0\times L_1\mid L_0\in\mL_i(C), L_1\in\mL_j(D) \} .
\end{equation}
For the skeletons of a direct product we have the following useful formula:
\begin{equation}\label{eq:Skel_k(CxD)=...}
  \Skel_k(C\times D) = \Disjunionu{i+j=k} \Skel_i(C)\times \Skel_j(D) .
\end{equation}
These formulas~\eqref{eq:L_k(CxD)=...} and~\eqref{eq:Skel_k(CxD)=...} are verified easily.

The final aspect we will address here concerns the largest linear subspace contained in the cone and the dimension of the linear span of the cone. The largest linear subspace contained in~$C$ is given by $C\cap (-C)$, and its dimension is known as the \emph{lineality} of the cone,
\begin{equation}\label{eq:def-lin(C)}
  \lin(C)=\dim(C\cap (-C)) .
\end{equation}
The dimension of the linear span is connected with the lineality via polarity, as it is easy to verify that
\begin{equation*}
   \dim(\spa(C)) + \lin(C^\polar) = d .
\end{equation*}
To avoid redundancy we therefore do not introduce a separate notation for the dimension of the linear span of a cone. Note that the lineality of a product is given by $\lin(C\times D)=\lin(C)+\lin(D)$.

It turns out that the lineality fits perfectly into our theory if we define the $\vct\ell$-vector of a cone in the following way: for $C\in\P(\IR^d)$,
\begin{align}\label{eq:def-ell(C)}
   \vct\ell(C) & = \big(\ell_0(C),\ldots,\ell_d(C)\big) , & \ell_k(C) & = \begin{cases} 1 & \text{if } \lin(C)=k \\ 0 & \text{else} . \end{cases}
\end{align}
Note that if we used the dimension of the linear span instead of the lineality for this vector construction, then the resulting vector would be the reversal of $\vct\ell(C^\polar)$. Note also that
\begin{equation}\label{eq:l_0(C)=f_0(C)=u_0(C)}
  \ell_0(C) = f_0(C) = u_0(C) \quad (\neq v_0(C) \text{ in general}) .
\end{equation}
The $\vct\ell$-vector (and its localization, to be introduced in Section~\ref{sec:prelim-conmeas} below) will play a role in the characterization of the polyhedral measures, cf.~Section~\ref{sec:char-polyh}.

We note that the $\vct\ell$-vector satisfies the properties (0)--(2) of the $\vct f$-vector, i.e., linear invariance: $\vct\ell(\vct TC)=\vct\ell(C)$ for all $\vct T\in\Gl_d$, subspace-property: $\vct\ell(L)=\vct f(L)=\vct u(L)=\vct v(L)$ for linear subspaces $L\subseteq\IR^d$, and product rule: $\ell_k(C\times D)=\sum_{i+j=k}\ell_i(C)\ell_j(C)$. It also satisfies a ``strong kinematic formula'', as shown in the following proposition, cp.~Theorem~\ref{thm:kinform-u,v}.

\begin{proposition}\label{prop:kin-form-ell-vect}
Let $C_0,\ldots,C_n\in\P(\IR^d)$ and $\vct T_0,\ldots,\vct T_n\in\Gl_d$. Then for almost all $(\vct Q_0,\ldots,\vct Q_n)\in O(d)^{n+1}$ and for $k>0$,
\begin{align}\label{eq:kinform-ell}
   \ell_k\bigg( \bigcap_{i=0}^n \vct T_i\vct Q_i C_i \bigg) & = \ell_{nd+k}(C_0\times\cdots\times C_n) , & \ell_0\bigg( \bigcap_{i=0}^n \vct T_i\vct Q_i C_i \bigg) & = \sum_{j=0}^{nd} \ell_j(C_0\times\cdots\times C_n) .
\end{align}
\end{proposition}

Note that the second equation in~\eqref{eq:kinform-ell} in connection with the observation~\eqref{eq:l_0(C)=f_0(C)=u_0(C)} settles the boundary case in~\eqref{eq:kinform-u}.

\begin{proof}[Proof of Proposition~\ref{prop:kin-form-ell-vect}]
Let $L_i:=C_i\cap (-C_i)$, denote the largest subspace contained in~$C_i$. Then the largest subspace contained in $C:=\bigcap_{i=0}^n \vct T_i\vct Q_i C_i$ is given by $C\cap (-C) = \bigcap_{i=0}^n \vct T_i\vct Q_i L_i$. The dimension of this intersection is almost surely given by $\max\{ 0, k_0 + \cdots + k_n - nd\}$, where $k_i:=\dim(L_i)=\lin(C_i)$, $i=0,\ldots,n$, cf.~Lemma~\ref{lem:TQC-alm-sure-generic}. Hence, if $k>0$, then almost surely $\ell_k\big( \bigcap_{i=0}^n \vct T_i\vct Q_i C_i \big)=1$ iff $nd+k=k_0+\cdots+k_n$ and zero else. The same holds for $\ell_{nd+k}(C_0\times\cdots\times C_n)$, as $\lin(C_0\times \cdots\times C_n)=k_0+\cdots+k_n$. This settles the case $k>0$. One argues analogously in the case $k=0$.
\end{proof}

\subsection{Preliminaries: conic Borel sets and measures}\label{sec:prelim-conmeas}

Given a conic Borel set $M\in\hmB(\IR^d)$, i.e., $\lambda M=M$ for all $\lambda>0$, almost all information about this set is contained in the intersection $M\cap S^{d-1}$ with the unit sphere, except for the information if the origin is contained in the set or not. This observation shows that~$\hmB(\IR^d)$ decomposes disjointly into:
\begin{equation}\label{eq:decomp-hmB(R^d)}
  \hmB(\IR^d) = \{ M\in\hmB(\IR^d)\mid \vct 0\not\in M\} \;\uplus\; \{ M\in\hmB(\IR^d)\mid \vct 0\in M\} ,
\end{equation}
where both parts are equivalent to the Borel algebra on the unit sphere, $\{ M\in\hmB(\IR^d)\mid \vct 0\not\in M\} \equiv \{ M\in\hmB(\IR^d)\mid \vct 0\in M\} \equiv \mB(S^{d-1})$. The following convention turns out to be particularly convenient: for $M\in\hmB(\IR^d)$,
  \[ M_* := M\setminus\{\vct0\} . \]
We denote the embedding of the set of spherical measures into the set of conic measures by
\begin{equation}\label{eq:conic-ext-spher}
  \nu\in\mM(S^{d-1})\mapsto \hat\nu\in\hmM(\IR^d) ,\qquad \hat\nu(M) := \nu(M\cap S^{d-1}) \text{ for } M\in\hmB(\IR^d) .
\end{equation}
Given a conic measure $\mu\in\hmM(\IR^d)$, the decomposition~\eqref{eq:decomp-hmB(R^d)} implies that
  \[ \mu-\mu(\{\vct0\})\Dir = \hat\nu \]
for some spherical measures $\nu\in\mM(S^{d-1})$. This shows that every conic measure can be written uniquely as the sum of a (lifted) spherical measure and a scaled Dirac measure.

\begin{lemma}\label{lem:orth-inv-conic-meas}
Let $\mu\in\hmM(\IR^d)$ be an orthogonal invariant conic measure, i.e., $\mu(\vct QM)=\mu(M)$ for all $M\in\hmB(\IR^d)$, $\vct Q\in O(d)$. Then
\begin{equation}\label{eq:dirac-gaussian-basis}
  \mu = \mu(\{\vct0\})\,\Dir + \mu(\IR^d_*) \gamma_d .
\end{equation}
\end{lemma}

\begin{proof}
If $\mu\in\hmM(\IR^d)$ then it can be written as $\mu=\mu(\{\vct0\})\Dir + \hat\nu$ for some spherical measure $\nu\in\mM(S^{d-1})$. The spherical measure~$\nu$ is orthogonal invariant, since
  \[ \nu(\vct Q\bar M) = \hat\nu(\vct QM) = \mu(\vct QM) = \mu(M) = \hat\nu(M) = \nu(\bar M) \]
for any spherical Borel set $\bar M\in\mB(S^{d-1})$ and the corresponding conic borel set $M=\{\lambda\vct x\mid \lambda>0 , \vct x\in\bar M\}$. The Lebesgue measure is up to scaling the only orthogonal invariant Borel measure on the sphere, see for example~\cite[Ch.~13]{SW:08}. This implies $\hat\nu = \mu(\IR^d_*)\gamma_d$.
\end{proof}

\subsection{Polyhedral and curvature measures}\label{sec:polyh-curv-meas}

Before considering the localizations of the $\vct u$- and $\vct v$-vectors, we make the Dirac measure~$\Dir$ supported at the origin cone dependent by defining for a function $h\colon\P(\IR^d)\to\IR$,
\begin{align}\label{eq:def-Dir_h}
  \Dir_h & \colon\P(\IR^d)\times \hmB(\IR^d)\to\IR , & \Dir_h(C,M) & := h(C)\,\Dir(M) .
\end{align}
An important special case is obtained by choosing for~$h$ the indicator function for the lineality of~$C$, which gives rise to the following definition.

\begin{definition}
For $0\leq k\leq d$ we define the $k$th \emph{lineality measure} $\Lin_k\colon\P(\IR^d)\times \hmB(\IR^d)\to\IR$,
\begin{align}\label{eq:def-Lin_k}
   \Lin_k(C,M) & := \Dir_{h_k}(C,M) \quad\text{with}\quad h_k(C) := \begin{cases} 1 & \text{if } \lin(C)=k \\ 0 & \text{if } \lin(C)\neq k . \end{cases}
\end{align}
\end{definition}

The lineality measures localize the $\ell$-vector, as $\Lin_k(C,\IR^d)=\ell_k(C)$. In fact, we even have $\Lin_k(C,M)=\ell_k(C)$ for all $M\in\hmB(\IR^d)$ such that $\vct0\in M$ (and $\Lin_k(C,M)=0$ if $\vct0\not\in M$). Due to this close connection, the lineality measures also inherit a number of easily seen properties, which we spare ourselves from enumerating here.
Finally, note that the lineality measures decompose the Dirac measure,
  \[ \Dir=\Lin_0(C,\cdot) + \cdots + \Lin_d(C,\cdot) . \]

The $\vct u$- and $\vct v$-vector are localized by the polyhedral measures and the curvature measures, respectively, and these are given by
\begin{align*}
   \Psi_k(C,M) & = \sum_{L\in\mL_k(C)} \gamma_L(C\cap L\cap M) ,
 & \Phi_k(C,M) & = \sum_{L\in\mL_k(C)} \gamma_L(C\cap L\cap M) \, \gamma_{L^\bot}(C^\polar\cap L^\bot) ,
\end{align*}
where $C\in\P(\IR^d)$ and $0\leq k\leq d$.

\begin{remark}
For $k=d$ the polyhedral measure coincides with the curvature measure, $\Psi_d(C,M)=\Phi_d(C,M)=\gamma_d(C\cap M)$. Furthermore, if $C\in\P(\IR^d)$ has nonempty interior and is not the whole space~$\IR^d$, then $\Phi_{d-1}(C,\cdot)=\frac{1}{2}\Psi_{d-1}(C,\cdot)$. At the other extreme, $2\Psi_1(C,M)$ counts the number of extreme rays of~$C$ lying in~$M$; the $0$th polyhedral measure coincides with the $0$th lineality measure, $\Psi_0=\Lin_0$, and the $0$th curvature measure is given by the Dirac measure scaled by the $0$th intrinsic volume, $\Phi_0=\Dir_{v_0}$, cf.~\eqref{eq:def-Dir_h}.
\end{remark}

These measures satisfy some fundamental properties, which, as we will show in Section~\ref{sec:char-polyh}, actually characterize these measures. We list these properties in the following proposition.

\begin{proposition}\label{prop:props-polyh-curv-meas}
Let $C\in\P(\IR^d)$ and $0\leq k\leq d$.
\begin{enumerate}
  \item \emph{concentration:} for all $M\in\hmB(\IR^d)$,
  \begin{align*}
     \Psi_k(C,M) & = \Psi_k\big(C,M\cap \Skel_k(C)\big) , & \Phi_k(C,M) & = \Phi_k\big(C,M\cap \Skel_k(C)\big) .
  \end{align*}
  \item \emph{orthogonal invariance:} for all $\vct Q\in O(d)$ and all $M\in\hmB(\IR^d)$,
  \begin{align*}
     \Psi_k(\vct Q C,\vct Q M) & = \Psi_k(C,M) , & \Phi_k(\vct Q C,\vct Q M) & = \Phi_k(C,M) .
  \end{align*}
  \item \emph{locality:} for all $D\in\P(\IR^d)$, $M\in\hmB(\IR^d)$, and all open sets $B\in\hmB(\IR^d)$,
  \begin{align*}
     \Psi_k(C,M) & = \Psi_k(D,M) , & \text{if } \;\; \Skel_k(C)\cap M & = \Skel_k(D)\cap M ,
  \\ \Phi_k(C,B) & = \Phi_k(D,B) , & \text{if } \;\; C\cap B & = D\cap B .
  \end{align*}
  \item \emph{product rule:} for all $D\in\P(\IR^e)$, $M\in\hmB(\IR^d)$, $N\in\hmB(\IR^e)$,
  \begin{align*}
     \Psi_k(C\times D,M\times N) & = \sum_{i+j=k} \Psi_i(C,M)\;\Psi_j(D,N) ,
  \\ \Phi_k(C\times D,M\times N) & = \sum_{i+j=k} \Phi_i(C,M)\;\Phi_j(D,N) ,
  \end{align*}
  \item \emph{additivity:} for all $D\in\P(\IR^d)$ such that $C+D = C\cup D$ and all $M\in\hmB(\IR^d)$,
  \begin{align*}
     \Phi_k(C+D,M) + \Phi_k(C\cap D,M) & = \Phi_k(C,M) + \Phi_k(D,M) .
  \end{align*}
  \item \emph{weak continuity:} for all open sets $B\in\hmB(\IR^d)$, 
  \begin{align*}
     \liminf_i \; \Phi_k(C_i,B) & \geq \Phi_k(C,B) , & \text{if } \;\; \big\{C_i\mid i\in\IN\big\}\subseteq\P(\IR^d) \text{ such that } \lim_{i\to\infty}C_i & = C .
  \end{align*}
\end{enumerate}
\end{proposition}

\begin{remark}
If $C\cap B=D\cap B$ for some open conic set $B\in\hmB(\IR^d)$, then $\Skel_k(C)\cap B = \Skel_k(D)\cap B$ for all $k=0,\ldots,d$. So the locality property of the polyhedral measures is stronger than the locality property of the curvature measures. On the other hand, except in the case $k\in\{0,d\}$, the polyhedral measures do not satisfy the additivity property, which is satisfied by the curvature measures, and the polyhedral measures are also not weakly continuous except in the cases $k\in\{0,d-1,d\}$.
\end{remark}

\begin{proof}
Properties~(1) and~(2) follow directly from the definitions of~$\Psi_k$ and~$\Phi_k$. As for the locality, assume first that $\Skel_k(C)\cap M = \Skel_k(D)\cap M$ for $k=0,\ldots,d$. In this case we obtain for every $L\in\mL_k(C)$: if $\inter_L(L\cap\Skel_k(C)\cap M)=\inter_L(L\cap\Skel_k(D)\cap M)\neq\emptyset$, then $L\in\mL_k(D)$; and the same holds of course with $C$ and $D$ exchanged. Therefore,
\begin{align*}
   \Psi_k(C,M) & = \sum_{\substack{L\in\mL_k(C)\\ L\cap \Skel_k(C)\cap M\neq\emptyset}} \gamma_L(L\cap \Skel_k(C)\cap M) = \sum_{\substack{L\in\mL_k(D)\\ L\cap \Skel_k(D)\cap M\neq\emptyset}} \gamma_L(L\cap \Skel_k(D)\cap M) = \Psi_k(D,M) .
\end{align*}
If $C\cap B=D\cap B$ for some open conic set $B\in\hmB(\IR^d)$, then $\Skel_k(C)\cap B = \Skel_k(D)\cap B$ for all $k=0,\ldots,d$. Furthermore, if $\Pi_C(\vct x)\in \Skel_k(C)\cap B=\Skel_k(D)\cap B$, then $\Pi_C(\vct x)=\Pi_D(\vct x)$; and the same holds with $C$~and $D$ exchanged. Therefore, using~\eqref{eq:def-Phi-vect-Prob},
\begin{align*}
   \Phi_k(C,B) & = \Prob\big\{\Pi_C(\vct g)\in\Skel_k(C)\cap B\big\} = \Prob\big\{\Pi_D(\vct g)\in\Skel_k(D)\cap B\big\} = \Phi_k(D,B) .
\end{align*}
The product rules follow directly from~\eqref{eq:L_k(CxD)=...} and~\eqref{eq:Skel_k(CxD)=...}.
The additivity and weak continuity of the curvature measures are special cases of the corresponding properties of the support measures, cf.~Proposition~\ref{prop:props-supp-meas}.
\end{proof}

\subsection{Characterization}\label{sec:char-polyh}

The following two theorems show that the properties of the polyhedral measures and the curvature measures listed in Proposition~\ref{prop:props-polyh-curv-meas} (basically) characterize these measures. In this section we will only prove the characterization theorem for the polyhedral measures, which seems to be new, and which may be regarded as the conic version of the characterization theorem for the support measures, which we will provide in Section~\ref{sec:char-supp-meas}. The characterization of the curvature measures is due to Schneider~\cite{Sch:78} and its proof, which can be found (in its spherical version) in~\cite[Thms.~6.5.4 \&~14.4.7]{SW:08}, is considerably more involved.

\begin{theorem}[Characterization of polyhedral measures]\label{thm:conic-char}
Let $\psi\colon\P(\IR^d)\times\hmB(\IR^d)\to\IR$ be such that for all $C,D\in\P(\IR^d)$, $M\in\hmB(\IR^d)$, $\vct Q\in O(d)$:
\begin{enumerate}\setcounter{enumi}{-1}
  \item $\psi(C,\cdot)\in\hmM(\IR^d)$,
  \item $\psi(C,M)=\psi(C,M\cap C)$,
  \item $\psi(\vct Q C,\vct Q M)=\psi(C,M)$,
  \item $\psi(C,M)=\psi(D,M)$ if $\Skel_k(C)\cap M=\Skel_k(D)\cap M$ for $k=0,\ldots,d$.
\end{enumerate}
Then $\psi$ is a linear combination of~$\Lin_d,\ldots,\Lin_1,\Lin_0=\Psi_0,\Psi_1,\ldots,\Psi_d$. In this case
\begin{equation}\label{eq:exp-psi-Psi}
  \psi = \sum_{k=0}^d \psi(\Lk,\{\vct0\})\,\Lin_k + \sum_{k=1}^d \psi(\Lk,\IR^d_*)\,\Psi_k ,
\end{equation}
where $\Lk\subseteq\IR^d$ denotes a $k$-dimensional subspace.
\end{theorem}

\begin{proof}
Let $\psi\colon\P(\IR^d)\times\hmB(\IR^d)\to\IR$ satisfy the assumptions (0)--(3), i.e., $\psi(C,\cdot)$ is a conic measure, which is concentrated on~$C$, $\psi$~is invariant under (simultaneous) orthogonal transformations, and $\psi(C,M)=\psi(D,M)$, if the skeletons of the cones~$C,D$ coincide in~$M$.

If $L\subseteq\IR^d$ is a linear subspace, then $\psi(L,\cdot)$ is an orthogonal invariant conic measure on~$L$, so that by Lemma~\ref{lem:orth-inv-conic-meas} it is a linear combination of the Dirac measure and the Gaussian measure~$\gamma_L$. Moreover, orthogonal invariance implies that, for $k=\dim L>0$,
\begin{equation}\label{eq:psi(L,.)=}
  \psi(L,\cdot) = \psi(\Lk,\{\vct0\})\,\Dir + \psi(\Lk,\IR^d_*) \gamma_L ,
\end{equation}
where $\Lk\subseteq\IR^d$ some $k$-dimensional linear subspace. For $C\in\P(\IR^d)$ we will use the facial decomposition~\eqref{eq:fac-decomp-C}. For $M\in\hmB(\IR^d)$, we obtain
\begin{align*}
   \psi(C,M) & \stackrel{(1)}{=} \psi(C,M\cap C) \stackrel{\eqref{eq:fac-decomp-C}}{=} \sum_{k=0}^d \; \sum_{L\in\mL_k(C)} \psi(C,M\cap \inter_L(C\cap L))
\\ & \stackrel{(3)}{=} \sum_{k=0}^d \sum_{L\in\mL_k(C)} \psi(L,M\cap \inter_L(C\cap L))
\\ & \stackrel{\eqref{eq:psi(L,.)=}}{=} \sum_{k=0}^d \psi(\Lk,\{\vct0\}) \sum_{L\in\mL_k(C)} \Dir(M\cap \inter_L(C\cap L)) + \sum_{k=1}^d \psi(\Lk,\IR^d_*) \sum_{L\in\mL_k(C)} \gamma_L(M\cap C\cap L)
\\ & \stackrel{\eqref{eq:def-Lin_k}/\eqref{eq:def-Psi_k}}{=} \sum_{k=0}^d \psi(\Lk,\{\vct0\})\, \Lin_k(C,M) + \sum_{k=1}^d \psi(\Lk,\IR^d_*)\, \Psi_k(C,M) . \qedhere
\end{align*}
\end{proof}

\begin{theorem}[Characterization of curvature measures]\label{thm:conic-char-curv}
Let $\psi\colon\P(\IR^d)\times\hmB(\IR^d)\to\IR_+$ be such that for all $C,D\in\P(\IR^d)$, $M\in\hmB(\IR^d)$, $\vct Q\in O(d)$, $B\in\hmB(\IR^d)$ open:
\begin{enumerate}\setcounter{enumi}{-1}
  \item $\psi(C,\cdot)\in\hmM(\IR^d)$,
  \item $\psi(C,M)=\psi(C,M\cap C)$,
  \item $\psi(\vct Q C,\vct Q M)=\psi(C,M)$,
  \item $\psi(C,B)=\psi(D,B)$ if $C\cap B=D\cap B$,
  \item $\psi(C+D,M) + \psi(C\cap D,M) = \psi(C,M) + \psi(D,M)$, if $C+D=C\cup D$.
\end{enumerate}
Then $\psi$ is a nonnegative linear combination of $\Phi_1,\ldots,\Phi_d$ and $\Dir_h$, where $h\colon\P(\IR^d)\to\IR_+$ is given by $h(C)=\psi(C,\{\vct0\})$. In this case
\begin{equation}\label{eq:exp-psi-Phi}
  \psi = \Dir_h + \sum_{k=1}^d \psi(\Lk,\IR^d_*)\,\Phi_k ,
\end{equation}
where $\Lk\subseteq\IR^d$ denotes a $k$-dimensional subspace.
\end{theorem}

\subsection{Kinematic formulas}\label{sec:calc-kinem-conic}

The following theorem provides a kinematic formula for the polyhedral measures, which specializes to the kinematic formula~\eqref{eq:kinform-u} for the $\vct u$-vector in the case $M_0=\cdots=M_n=\IR^d$. So as a corollary we obtain the first half of Theorem~\ref{thm:kinform-u,v}.

\begin{theorem}[Kinematic formula for polyhedral measures]\label{thm:kinform-polyh}
Let $C_0,\ldots,C_n\in\P(\IR^d)$, $M_0,\ldots,M_n\in\hmB(\IR^d)$, $\vct T_0,\ldots,\vct T_n\in\Gl_d$. Then for $\vct Q_0,\ldots,\vct Q_n\in O(d)$ iid uniformly at random and $k>0$,
\begin{equation}\label{eq:kinform-polyh}
   \Expect\bigg[ \Psi_k\bigg( \bigcap_{i=0}^n \vct T_i \vct Q_i C_i , \bigcap_{i=0}^n \vct T_i \vct Q_i M_i\bigg)\bigg] = \Psi_{nd+k}(C,M) ,
\end{equation}
where $C:=C_0\times\cdots\times C_n$, $M:=M_0\times \cdots\times M_n$, and where the polyhedral measures of the product are given by
\begin{equation}\label{eq:Psi_m(CxD)=...}
  \Psi_m(C,M) = \sum_{i_0+\cdots+i_n=m} \Psi_{i_0}(C_0,M_0)\cdots \Psi_{i_n}(C_n,M_n) .
\end{equation}
\end{theorem}

The product formula~\eqref{eq:Psi_m(CxD)=...} is of course just an iteration of~(4) in Proposition~\ref{prop:props-polyh-curv-meas}. For the proof of Theorem~\ref{thm:kinform-polyh} we need a few of lemmas about properties of generic intersections of cones. In fact, at first glance it might not even be clear that the expectation in~\eqref{eq:kinform-polyh} exists (see~(1) in Proposition~\ref{prop:genericity}). As some of these lemmas are of technical nature while their statements are geometrically obvious, we defer their proofs to Appendix~\ref{sec:genericity}.

\begin{lemma}\label{lem:psi-satisf-assumpts}
Let $C_0,\ldots,C_n\in\P(\IR^d)$, $M_0,\ldots,M_n\in\hmB(\IR^d)$, $\vct T_0,\ldots,\vct T_n\in\Gl_d$, and $k>0$. Then for $\vct Q_0,\ldots,\vct Q_n\in O(d)$ iid uniformly at random, the map $\psi\colon\P(\IR^d)\times\hmB(\IR^d)\to\IR$ given by
\begin{equation}\label{eq:kinform-polyh_lem}
   \psi(C_0,M_0) = \Expect\bigg[ \Psi_k\bigg( \bigcap_{i=0}^n \vct T_i \vct Q_i C_i , \bigcap_{i=0}^n \vct T_i \vct Q_i M_i\bigg)\bigg]
\end{equation}
satisfies the assumptions (0)--(3) in Theorem~\ref{thm:conic-char}.
\end{lemma}

\begin{proof}
To simplify the notation we abbreviate $\vct T_i\vct Q_i =: \vct U_i$.

(0) \emph{Claim:} $\psi(C_0,\cdot)\in\hmM(\IR^d)$. For pairwise disjoint $N_j\in\hmB(\IR^d)$, $j=1,2,\ldots$, we have that $\vct U_0N_j$, $j=1,2,\ldots$, and $\vct U_0N_j\cap \bigcap_{i=1}^n \vct U_i M_i$, $j=1,2,\ldots$, are pairwise disjoint as well. Therefore,
\begin{align*}
   & \psi\bigg(C_0,\bigcup_{j=1}^\infty N_j\bigg) = \Expect\bigg[ \Psi_k\bigg( \bigcap_{i=0}^n \vct U_i C_i , \bigg(\bigcup_{j=1}^\infty \vct U_0N_j\bigg) \cap \bigcap_{i=1}^n \vct U_i M_i \bigg)\bigg] = \Expect\bigg[ \Psi_k\bigg( \bigcap_{i=0}^n \vct U_i C_i , \bigcup_{j=1}^\infty \bigg(\vct U_0N_j \cap \bigcap_{i=1}^n \vct U_i M_i \bigg)\bigg)\bigg]
\\ &  \;\; = \Expect\bigg[ \sum_{j=1}^\infty \Psi_k\bigg( \bigcap_{i=0}^n \vct U_i C_i , \vct U_0N_j \cap \bigcap_{i=1}^n \vct U_i M_i \bigg) \bigg] \stackrel{(*)}{=} \sum_{j=1}^\infty \Expect\bigg[ \Psi_k\bigg( \bigcap_{i=0}^n \vct U_i C_i , \vct U_0N_j \cap \bigcap_{i=1}^n \vct U_i M_i \bigg) \bigg] = \sum_{j=1}^\infty \psi(C_0,N_j) ,
\end{align*}
where $(*)$ follows from an application of the monotone convergence theorem.

(1) \emph{Claim:} $\psi(C_0,M_0)=\psi(C_0,M_0\cap C_0)$. This follows directly from the locality of~$\Psi_k$,
\begin{align*}
   \psi(C_0,M_0\cap C_0) & = \Expect\bigg[ \Psi_k\bigg( \bigcap_{i=0}^n \vct U_i C_i , \vct U_0 C_0 \cap \bigcap_{i=0}^n \vct U_i M_i \bigg)\bigg] = \Expect\bigg[ \Psi_k\bigg( \bigcap_{i=0}^n \vct U_i C_i , \bigcap_{i=0}^n \vct U_i M_i \bigg)\bigg] = \psi(C_0,M_0) .
\end{align*}

(2) \emph{Claim:} $\psi(\vct Q C_0,\vct Q M_0)=\psi(C_0,M_0)$ for $\vct Q\in O(d)$. This follows from the observation that $\vct Q_0\vct Q$ is uniformly at random in~$O(d)$ and independent of $\vct Q_1,\ldots,\vct Q_n$.

(3) \emph{Claim:} $\psi(C_0,M_0)=\psi(\tilde C_0,M_0)$ if $\Skel_k(C_0)\cap M_0=\Skel_k(\tilde C_0)\cap M_0$ for all $k=0,\ldots,d$. See (2) in Proposition~\ref{prop:genericity}.
\end{proof}

The following simple lemma will be convenient in the proof of Theorem~\ref{thm:kinform-polyh}.

\begin{lemma}\label{lem:prod-exch}
Let $k>0$, $C,D\in\P(\IR^d)$, $M,N\in\hmB(\IR^d)$. Then
  \[ \sum_{j=1}^d \Psi_j(C,M) \, \Psi_{d+k}(L_j\times D,\IR^d\times N) = \Psi_{d+k}(C\times D,M\times N) , \]
where $L_j\subseteq\IR^d$ a $j$-dimensional linear subspace.
\end{lemma}

\begin{proof}
Using the product rule in Proposition~\ref{prop:props-polyh-curv-meas}, we obtain
\begin{align*}
   \Psi_{d+k}(L_j\times D,\IR^d\times N) & = \sum_{\ell+m=d+k} \Psi_\ell(L_j,\IR^d)\,\Psi_m(D,N) = \Psi_{d+k-j}(D,N) ,
\intertext{and therefore, using the product rule one more time,}
   \sum_{j=1}^d \Psi_j(C,M) \, \Psi_{d+k}(L_j\times D,\IR^d\times N) & = \sum_{j=1}^d \Psi_j(C,M) \, \Psi_{d+k-j}(D,N) = \Psi_{d+k}(C\times D,M\times N) . \qedhere
\end{align*}
\end{proof}

\begin{proof}[Proof of Theorem~\ref{thm:kinform-polyh}]
To simplify the notation we abbreviate $\vct T_i\vct Q_i =: \vct U_i$. Define
\begin{align*}
   \LEFT(C_0,\ldots,C_n;M_0,\ldots,M_n) & := \Expect\bigg[ \Psi_k\bigg( \bigcap_{i=0}^n \vct U_i C_i , \bigcap_{i=0}^n \vct U_i M_i\bigg)\bigg] ,
\\ \RIGHT(C_0,\ldots,C_n;M_0,\ldots,M_n) & := \Psi_{nd+k}\big( C_0\times\cdots\times C_n ,\, M_0\times \cdots\times M_n \big) ,
\end{align*}
so that we need to show $\LEFT(C_0,\ldots,C_n;M_0,\ldots,M_n) = \RIGHT(C_0,\ldots,C_n;M_0,\ldots,M_n)$. By induction on~$m$ we will show that
\begin{align*}
   \LEFT\big(C_0,\ldots,C_{m-1},L^{(0)},\ldots,L^{(n-m)};M_0,\ldots,M_{m-1},\IR^d,\ldots,\IR^d\big) &
\\ = \RIGHT\big(C_0,\ldots,C_{m-1},L^{(0)},\ldots,L^{(n-m)};M_0,\ldots,M_{m-1},\IR^d,\ldots,\IR^d\big) &
\end{align*}
where $L^{(0)},\ldots,L^{(n-m)}\subseteq\IR^d$ linear subspaces.

The case $m=0$ is easily established: the intersection $\bigcap_{i=0}^n \vct U_i L^{(i)}$ is almost surely a linear subspace of dimension $\max\{\dim L^{(0)}+\cdots+\dim L^{(n)}-nd,0\}$, and the direct product $L^{(0)}\times\cdots\times L^{(n)}$ is a linear subspace of dimension $\dim L^{(0)}+\cdots+\dim L^{(n)}$, so that
\begin{align*}
   \LEFT\big(L^{(0)},\ldots,L^{(n)};\IR^d,\ldots,\IR^d\big) & = \begin{cases} 1 & \text{if } nd+k = \dim L^{(0)}+\cdots+\dim L^{(n)} \\ 0 & \text{else} \end{cases}
\\ & = \RIGHT\big(L^{(0)},\ldots,L^{(n)};\IR^d,\ldots,\IR^d\big) .
\end{align*}

For the induction step, $m\geq1$, we define $\psi\colon\P(\IR^d)\times\hmB(\IR^d)\to\IR$,
\begin{align*}
   \psi(C,M) & := \LEFT\big(C_0,\ldots,C_{m-2},C,L^{(0)},\ldots,L^{(n-m)};M_0,\ldots,M_{m-2},M,\IR^d,\ldots,\IR^d)
\\ & = \Expect\bigg[ \Psi_k\bigg( \bigcap_{i=0}^{m-2} \vct U_i C_i \cap \vct U C\cap \bigcap_{i=0}^{n-m} \vct U_{m+i} L^{(i)} , \bigcap_{i=0}^{m-2} \vct U_i M_i \cap \vct U M \bigg)\bigg] ,
\end{align*}
where we set $\vct U:=\vct U_{m-1}$ to simplify the notation. Note that $\Psi_k(C,\{\vct0\})=0$, since $k>0$, and thus $\psi(C,\{\vct0\})=0$. By Lemma~\ref{lem:psi-satisf-assumpts}, $\psi$ satisfies the assumptions in Theorem~\ref{thm:conic-char}, so that
\begin{align*}
   & \LEFT\big(C_0,\ldots,C_{m-1},L^{(0)},\ldots,L^{(n-m)};M_0,\ldots,M_{m-1},\IR^d,\ldots,\IR^d\big)
\\ & = \psi(C_{m-1},M_{m-1}) \stackrel{\eqref{eq:exp-psi-Psi}}{=} \sum_{j=0}^d \underbrace{\psi(L_j,\{\vct0\})}_{=0}\,\Lin_j(C_{m-1},M_{m-1}) + \sum_{j=1}^d \underbrace{\psi(L_j,\IR^d_*)}_{=\psi(L_j,\IR^d)}\,\Psi_j(C_{m-1},M_{m-1})
\\ & = \sum_{j=1}^d \LEFT\big(C_0,\ldots,C_{m-2},L_j,L^{(0)},\ldots,L^{(n-m)};M_0,\ldots,M_{m-2},\IR^d,\ldots,\IR^d\big)\,\Psi_j(C_{m-1},M_{m-1})
\\ & \stackrel{\text{(IH)}}{=} \sum_{j=1}^d \Psi_{nd+k}\big( C_0\times\cdots\times C_{m-2}\times L_j\times L^{(0)}\times \cdots\times L^{(n-m)} ,\, M_0\times \cdots\times M_{m-2}\times \IR^d\times \cdots\times \IR^d\big)
\\ & \hspace{15mm} \cdot \Psi_j(C_{m-1},M_{m-1})
\\ & \stackrel{\text{[Lem.~\ref{lem:prod-exch}]}}{=} \RIGHT\big(C_0,\ldots,C_{m-1},L^{(0)},\ldots,L^{(n-m)};M_0,\ldots,M_{m-1},\IR^d,\ldots,\IR^d\big) ,
\end{align*}
which shows the induction step, and thus finishes the proof.
\end{proof}

\section{Biconic localizations}\label{sec:biconic-localizations}

In this section we consider the biconic localizations of the intrinsic volumes, the support measures. As in Section~\ref{sec:conic-localizations} we start with some preliminaries on the (biconic) Borel algebra in Section~\ref{prelim:biconic} and with a discussion of the general properties of the support measures in Section~\ref{sec:suppmeas}. In Section~\ref{sec:char-supp-meas} we show how the support measures can be characterized through concentration, invariance, locality, and in Section~\ref{sec:kinform-suppmeas} we will use this characterization to prove a kinematic formula for the support measures, which generalizes~\eqref{eq:kinform-v}.

\subsection{Preliminaries: biconic Borel sets and measures}\label{prelim:biconic}

Recall that the biconic (Borel) $\sigma$-algebra on~$\IR^{d+d}=\IR^d\times\IR^d$ is defined by
\begin{align*}
  \hmB(\IR^d,\IR^d) & = \{ \M\subseteq\IR^{d+d} \text{ Borel set}\mid (\lambda,\lambda') \M=\M \text{ for all }\lambda,\lambda'>0 \} ,
\end{align*}
where $(\lambda,\lambda')\M=\{(\lambda\vct x,\lambda'\vct x')\mid (\vct x,\vct x')\in\M\}$. We denote the corresponding set of biconic measures by~$\hmM(\IR^d,\IR^d)$. The (biconic) Dirac measure on $\IR^{d+d}$ supported in $(\vct0,\vct0)$ is denoted by~$\mDir$.

Sets of the form $M\times M'$ with $M,M'\in\hmB(\IR^d)$ belong to~$\hmB(\IR^d,\IR^d)$, but not all elements in~$\hmB(\IR^d,\IR^d)$ are of this form. However, direct products \emph{generate} the biconic $\sigma$-algebra, as we will show in Proposition~\ref{prop:bicon-prod} below that the biconic $\sigma$-algebra is the product of the conic $\sigma$-algebras,
  \[ \hmB(\IR^d,\IR^d) = \hmB(\IR^d)\otimes \hmB(\IR^d) . \]

The set of polyhedral cones can be embedded into the biconic $\sigma$-algebra by the map
\begin{align}\label{eq:def-BL(C)}
   \BL & \colon \P(\IR^d) \to \hmB(\IR^d,\IR^d) , & \BL(C) & := \{(\vct x,\vct z)\in C\times C^\polar \mid \langle \vct x,\vct z\rangle = 0\} .
\end{align}
We call $\BL(C)$ the \emph{biconic lift} of~$C$. Note that for a linear subspace $L\subseteq\IR^d$ we obtain $\BL(L)=L\times L^\bot$, but in general $\BL(C)$ is not a direct product. Combining the facial decomposition~\eqref{eq:fac-decomp-C} with the polarity relation~\eqref{eq:L_k(C)--L_(d-k)(C^polar)} yields the following two disjoint decompositions of the biconic lift
\begin{align}\label{eq:decomp-Nor}
  \BL(C) & = \Disjunionu{L\in\mL(C)} \big(\inter_L(C\cap L)\times (C^\polar\cap L^\bot)\big) = \Disjunionu{L\in\mL(C)} \big((C\cap L)\times \inter_{L^\bot}(C^\polar\cap L^\bot)\big) .
\end{align}
These decompositions will be important for the proof of the Characterization Theorem~\ref{thm:biconic-char} in Section~\ref{sec:char-supp-meas}.

Besides lifting the whole cone into the biconic $\sigma$-algebra, it is also convenient to lift the $k$-skeleton via
\begin{align}\label{eq:def-lift-Skel_k(C)}
   \mSkel_k(C) & := \Disjunionu{L\in\mL_k(C)}\big( \inter_L(C\cap L)\times\inter_{L^\bot}(C^\polar\cap L^\bot)\big) , & \BLs(C) & := \Disjunionuo{k=0}{d} \mSkel_k(C) .
\end{align}
Except in the case where $C$ is a linear subspace, $\BLs(C)$ is a proper subset of~$\BL(C)$. But the following simple proposition shows that $\BLs(C)$ makes up the ``essential part'' of~$\BL(C)$.

\begin{proposition}\label{prop:biconic-proj}
Let $C\in\P(\IR^d)$ and let its projection map be denoted by $\Pi_C\colon\IR^d\to C$, $\Pi_C(\vct x)=\argmin\{\|\vct x-\vct y\|\mid \vct y\in C\}$. Then for every $\vct x\in\IR^d$,
\begin{equation}\label{eq:def-biconic-proj}
  \bproj_C(\vct x) := \big( \Pi_C(\vct x),\Pi_{C^\polar}(\vct x)\big) \in \BL(C) .
\end{equation}
Moreover, if $\vct g\in\IR^d$ denotes a Gaussian random vector, then almost surely
  \[ \bproj_C(\vct g) \in \BLs(C) . \]
\end{proposition}

\begin{proof}
The first claim is a well-known result by J.-J.~Moreau, cf.~for example~\cite[Sec.~3.2]{HUL:01}: the projection onto the primal cone~$\Pi_C(\vct x)=:\vct y$ and the projection onto the polar cone~$\Pi_{C^\polar}(\vct x)=:\vct y'$ satisfy $\langle \vct y,\vct y'\rangle=0$ (and $\vct y+\vct y'=\vct x$).

As for the second claim, it is easily seen that if $\big( \Pi_C(\vct x),\Pi_{C^\polar}(\vct x)\big)$ does not lie in the union of the $\mSkel_k(C)$, $k=0,\ldots,d$, then $\vct x$ lies in a subspace of the form $L+L'$ with $L\in\mL_m(C)$, $L'\in\mL_n(C^\polar)$, and $m+n<d$. The union of all these (finitely many) subspaces, has Gaussian measure zero, which shows the second claim.
\end{proof}

Proposition~\ref{prop:biconic-proj} implies that, using the notation from~\eqref{eq:def-biconic-proj},  we can write the support measures in the following way
\begin{equation}\label{eq:def-Theta-vect-Prob-Skel}
  \Theta_k(C,\M) = \Prob\big\{\bproj_C(\vct g)\in\mSkel_k(C)\cap\M\big\} ,
\end{equation}
cp.~the analogous characterizations of the intrinsic volumes~\eqref{eq:def-v-vect-Prob} and the curvature measures~\eqref{eq:def-Phi-vect-Prob}. In fact, Moreau's decomposition theorem states that
\begin{equation}\label{eq:addoPi_C=Id}
   \add\circ\bproj_C = \Id_d ,\qquad \add\colon \IR^{d+d}\to\IR^d ,\quad \add(\vct x,\vct x') := \vct x+\vct x' .
\end{equation}
Using this notation, we can write the support measures in the following form in which the projection map is completely eliminated,
\begin{equation}\label{eq:def-Theta-vect-Prob-no-proj}
  \Theta_k(C,\M) = \gamma_d\Big( \add\big( \mSkel_k(C)\cap\M\big) \Big) .
\end{equation}

The biconic structure naturally admits an involution, which we call the \emph{reversal map},
\begin{align}\label{eq:def-rev}
  \rev & \colon\hmB(\IR^d,\IR^d)\to\hmB(\IR^d,\IR^d) , & \rev(\M) & := \big\{(\vct x',\vct x)\mid (\vct x,\vct x')\in\M\big\} .
\end{align}
The reversal of the biconic lift is the biconic lift of the polar, $\rev\big( \BL(C)\big) = \BL(C^\polar)$, and for the lifted $k$-skeletons we obtain $\rev\big( \mSkel_k(C)\big) = \mSkel_{d-k}(C^\polar)$. Composing the reversal map with the biconic projection~\eqref{eq:def-biconic-proj} yields $\rev\circ\bproj_C=\bproj_{C^\polar}$.

Another natural definition is the following action of the general linear group. Recall that $(\vct TC)^\polar = \vct T^\invadj C^\polar$. We define the action of the general linear group on the biconic $\sigma$-algebra via
\begin{equation}\label{eq:def-TM-biconic}
  \vct T\M := \{(\vct{Tx},\vct T^\invadj\vct x')\mid (\vct x,\vct x')\in\M\} .
\end{equation}
This action has the following relations to the other structures we have introduced so far,
\begin{align*}
   \BL(\vct TC) & = \vct T\BL(C) , & \mSkel_k(\vct TC) & = \vct T\mSkel_k(C) , & \rev(\vct T\M) & = \vct T^\invadj\rev(\M) , & \vct T\bproj_C & = \bproj_{\vct TC} .
\end{align*}

When forming the product of biconic sets of the form $M\times M'$ and $N\times N'$ it makes sense to take the product of the first components as the first component and the products of the second componensts as the second component, i.e., take $M\times N\times M'\times N'$. We call the corresponding construction for general biconic sets the \emph{biconic product}: for $\M\in\hmB(\IR^d,\IR^d)$, $\N\in\hmB(\IR^e,\IR^e)$,
\begin{equation}\label{eq:bicon-prod}
  \M \bprod \N := \big\{ ( \vct x,\vct y,\vct x',\vct y') \mid (\vct x,\vct x')\in\M,(\vct y,\vct y')\in\N\big\} \in \hmB(\IR^{d+e},\IR^{d+e}) .
\end{equation}
Note that we indeed have $(M\times M')\bprod(N\times N')=M\times N\times M'\times N'$, and
\begin{align}\label{eq:biconic-prod-idents}
   \BL(C\times D) & = \BL(C)\bprod\BL(D) , & \mSkel_k(C\times D) & = \Disjunionu{i+j=k}\mSkel_i(C)\bprod\mSkel_j(D) , & \bproj_{C\times D} & = \bproj_C\bprod\bproj_D ,
\nonumber
\\ \rev(\M\bprod\N) & = \rev(\M)\bprod\rev(\N) , & \vct T(\M\bprod\N) & = \vct T\M \bprod \vct T\N .
\end{align}

The final structure, which we introduce on $\hmB(\IR^d,\IR^d)$, is the biconic version of intersection, and its associated reverse operation; we use the neutral terms of \emph{conjunction}~$\wedge$ and \emph{disjunction}~$\vee$: for $\M,\N\in\hmB(\IR^d,\IR^d)$,
\begin{align}
   \M\wedge\N & := \big\{(\vct x,\vct x'+\vct y')\mid (\vct x,\vct x')\in\M , (\vct x,\vct y')\in\N\big\} ,
\label{eq:def-conj-biconic}
\\ \M\vee\N & := \big\{(\vct x+\vct y,\vct x')\mid (\vct x,\vct x')\in\M , (\vct y,\vct x')\in\N\big\} .
\label{eq:def-disj-biconic}
\end{align}
For direct products $\M=M\times M'$, $\N=N\times N'$ we obtain
\begin{align*}
   \M\wedge\N & = (M\cap N)\times (M'+N') , & \M\vee\N & = (M+N)\times (M'\cap N') .
\end{align*}
The conjunction and disjunction naturally extend the lattice structure from $\P(\IR^d)$ to the biconic $\sigma$-algebra, since
\begin{align*}
   \BL(C)\wedge \BL(D) & = \BL(C\cap D) , & \BL(C)\vee \BL(D) & = \BL(C+D) .
\end{align*}
However, it should be noted that although the biconic $\sigma$-algebra $\hmB(\IR^d,\IR^d)$ is now endowed with a similar number of operations, $(\rev,\wedge,\vee)$, as the set of polyhedral cones~$\P(\IR^d)$, $(.^\polar,\cap,+)$, the structure of $\hmB(\IR^d,\IR^d)$ is \emph{significantly weaker} than the structure of~$\P(\IR^d)$. In fact, for $d\geq2$ the biconic $\sigma$-algebra even fails to be a lattice, as, for example, the idempotency axiom $\M\wedge\M=\M$ is in general not satisfied. The following proposition lists some further important properties of this structure, like De Morgan's Law.

\begin{proposition}\label{prop:biconic-conj}
Let $\M,\N,\M_0,\M_1\in\hmB(\IR^d,\IR^d)$. Then
\begin{align}
  \rev(\M\wedge \N) & = \rev(\M)\vee \rev(\N) , & (\M_0\cup\M_1) \wedge \N & = (\M_0\wedge\N) \cup (\M_1\wedge\N) ,
\label{eq:De-Morg-biconic}
\\ & & (\M_0\cap\M_1) \wedge \N & \subseteq (\M_0\wedge\N) \cap (\M_1\wedge\N) .
\nonumber
\end{align}
If $d\geq2$, then in general, $(\M_0\cap\M_1) \wedge \M \neq (\M_0\wedge\N) \cap (\M_1\wedge\N)$, and $\M_0\cap\M_1=\emptyset$ does not imply $(\M_0\wedge\N) \cap (\M_1\wedge\N)=\emptyset$.
\end{proposition}

\begin{proof}
De Morgan's Law follows directly from the definitions of~$\wedge,\vee,\rev$. Furthermore, we have
\begin{align*}
   (\M_0\cup\M_1) \wedge \N & = \big\{(\vct x,\vct x'+\vct y')\mid (\vct x,\vct x')\in\M_0\cup\M_1 , (\vct x,\vct y')\in\N\big\}
\\ & = \big\{(\vct x,\vct x'+\vct y')\mid (\vct x,\vct x')\in\M_0 , (\vct x,\vct y')\in\N\big\}\cup \big\{(\vct x,\vct x'+\vct y')\mid (\vct x,\vct x')\in\M_1 , (\vct x,\vct y')\in\N\big\}
\\ & = (\M_0\wedge\N) \cup (\M_1\wedge\N) ,
\\ (\M_0\cap\M_1) \wedge \N & = \big\{(\vct x,\vct x'+\vct y')\mid (\vct x,\vct x')\in\M_0\cap\M_1 , (\vct x,\vct y')\in\N\big\}
\\ & \subseteq \big\{(\vct x,\vct x'+\vct y')\mid (\vct x,\vct x')\in\M_0 , (\vct x,\vct y')\in\N\big\}\cap \big\{(\vct x,\vct x'+\vct y')\mid (\vct x,\vct x')\in\M_1 , (\vct x,\vct y')\in\N\big\}
\\ & = (\M_0\wedge\N) \cap (\M_1\wedge\N) .
\end{align*}
If $\M_0=M_0\times M_0'$, $\M_1=M_1\times M_1'$, $\N=N\times N'$, then
\begin{align*}
   (\M_0\cap\M_1) \wedge \N & = (M_0\cap M_1\cap N)\times \big((M_0'\cap M_1')+N\big)
\\ & \stackrel{(*)}{\neq} (M_0\cap M_1\cap N)\times \big((M_0'+N') \cap (M_1'+N')\big) = (\M_0 \wedge \N)\cap(\M_1 \wedge\N) ,
\end{align*}
where the inequality $(*)$ holds for example in the case where $M_0=N_1=N\neq\emptyset$ and $M_0',M_1',N'$ are pairwise linear independent lines lying in a $2$-dimensional linear space $L$, in which case $(M_0'\cap M_1')+N'=N'\neq L=(M_0'+N') \cap (M_1'+N')$. Replacing in the above example $M_0',M_1',N'$ by $M_0'\setminus\{\vct0\},M_1'\setminus\{\vct0\},N'\setminus\{\vct0\}$ yields $\M_0\cap\M_1=\emptyset$ but $(\M_0\wedge\N) \cap (\M_1\wedge\N)=M\times (L\setminus\{\vct0\})\neq\emptyset$.
\end{proof}

We finish this section with the announced result about the product structure of the biconic $\sigma$-algebra.

\begin{proposition}\label{prop:bicon-prod}
The biconic $\sigma$-algebra $\hmB(\IR^d,\IR^d)$ is the product algebra
\begin{equation}\label{eq:B(E,E')=B(E)xB(E')}
  \hmB(\IR^d,\IR^d)=\hmB(\IR^d)\otimes\hmB(\IR^d) .
\end{equation}
\end{proposition}

From this proposition we can deduce the following useful lemma about biconic measures.

\begin{lemma}\label{lem:coinc-dir-prod}
If $\mu_1,\mu_2\in\hmM(\IR^d,\IR^d)$ are such that $\mu_1(M\times M')=\mu_2(M\times M')$ for all $M,M'\in\hmB(\IR^d)$, then $\mu_1=\mu_2$.
\end{lemma}

\begin{proof}
The class $\{M\times M'\mid M,M'\in\hmB(\IR^d)\}$ is closed under finite intersections. If two measures $\mu_1,\mu_2\in\hmM(\IR^d,\IR^d)$ coincide on such a class, then they also coincide on the $\sigma$-algebra generated by it, cf.~\cite[Lem.~1.9.4]{Bog:07}, which is $\hmB(\IR^d,\IR^d)$ by Proposition~\ref{prop:bicon-prod}.
\end{proof}

\begin{proof}[Proof of Proposition~\ref{prop:bicon-prod}]
The inclusion $\supseteq$ in~\eqref{eq:B(E,E')=B(E)xB(E')} follows from the fact that $\hmB(\IR^d,\IR^d)$ is a $\sigma$-algebra, that contains all sets of the form $M\times M'$ with $M,M'\in\hmB(\IR^d)$. For the other inclusion let $\M\in\hmB(\IR^d,\IR^d)$ and decompose it into
\begin{align*}
   \M & = \M_* \;\uplus\; \M_1 \;\uplus\; \M_2 , &  \M_* & := \M\cap \big(\IR^d_*\times\IR^d_*\big) , 
\\ \M_1 & := \M\cap\big(\IR^d\times\{\vct0\}\big) , & \M_2 & := \M\cap\big(\{\vct0\}\times\IR^d_*\big) .
\end{align*}
Clearly, $\M_1,\M_2\in \hmB(\IR^d)\otimes\hmB(\IR^d)$, as both are of the form $M\times M'$ with $M,M'\in\hmB(\IR^d)$.
As for the remaining set $\M_*$, we make the following definition
\begin{align}
   \hmB^*(\IR^d) & := \{M\in\hmB(\IR^d)\mid \M\subseteq\IR^d_*\} ,
 & \hmB^*(\IR^d,\IR^d) & := \{\M\in\hmB(\IR^d,\IR^d)\mid \M\subseteq\IR^d_*\times\IR^d_*\}
\label{eq:def-B^*(E,E)}
\end{align}
Using the functions $\IR^d_*\to S^{d-1}$, $\vct x\mapsto\frac{\vct x}{\|\vct x\|}$, and $\IR^d_*\times\IR^d_*\to S^{d-1}\times S^{d-1}$, $(\vct x,\vct x')\mapsto(\frac{\vct x}{\|\vct x\|},\frac{\vct x'}{\|\vct x'\|})$, we see that
  \[ \hmB^*(\IR^d)\equiv \mB(S^{d-1}) ,\qquad \hmB^*(\IR^d,\IR^d)\equiv \mB(S^{d-1}\times S^{d-1}) . \]
We have $\mB(S^{d-1}\times S^{d-1})=\mB(S^{d-1})\otimes\mB(S^{d-1})$, cf.~\cite[Lem.~6.4.2]{Bog:07}, so that 
 \[ \hmB^*(\IR^d,\IR^d)\equiv \mB(S^{d-1}\times S^{d-1})=\mB(S^{d-1})\otimes\mB(S^{d-1})\equiv \hmB^*(\IR^d)\otimes\hmB^*(\IR^d)\subset \hmB(\IR^d)\otimes\hmB(\IR^d), \]
This shows that $\M_*\in\hmB(\IR^d)\otimes\hmB(\IR^d)$, which finishes the proof.
\end{proof}

\subsection{Support measures}\label{sec:suppmeas}

Recall that the support measures localize both the $\vct v$-vector and the curvature measures, and are given by $\Theta_k(C,\M) = \Prob\big\{\bproj_C(\vct g)\in\mSkel_k(C)\cap\M\big\}$ where $\vct g\in\IR^d$ is a standard Gaussian vector, cf.~\eqref{eq:def-Theta-vect-Prob-Skel}. By~\eqref{eq:def-Theta_k-dirprod} and Lemma~\ref{lem:coinc-dir-prod} we can write the support measures in the form
\begin{align*}
   \Theta_k(C,\cdot) & = \sum_{L\in\mL_k(C)} \gamma_L(C\cap L\cap \cdot)\otimes\gamma_{L^\bot}(C^\polar\cap L^\bot\cap\cdot) = \sum_{L\in\mL_k(C)} \Psi_k(C\cap L,\cdot)\otimes\Psi_{d-k}(C^\polar\cap L^\bot,\cdot) .
\end{align*}
In particular, we have (note that $\Psi_d=\Phi_d$ and $\lin(C^\polar)=0$ iff $\gamma_d(C)\neq0$)
\begin{align*}
   \Theta_d(C,\cdot) & = \Phi_d(C,\cdot) \otimes \Dir , & \Theta_0(C,\cdot) & = \Dir \otimes \Phi_d(C^\polar,\cdot) .
\end{align*}
The following proposition lists further properties of the support measures, similar to Proposition~\ref{prop:props-polyh-curv-meas}.

\begin{proposition}\label{prop:props-supp-meas}
Let $C,D\in\P(\IR^d)$ and $0\leq k\leq d$.
\begin{enumerate}
  \item \emph{concentration:} for all $\M\in\hmB(\IR^d,\IR^d)$,
  \begin{align*}
     \Theta_k(C,\M) & = \Theta_k\big(C,\M\cap \mSkel_k(C)\big) .
  \end{align*}
        In particular, $\Theta_k(C,\M) = \Theta_k\big(C,\M\cap \BL(C)\big)$.
  \item \emph{orthogonal invariance:} for all $\vct Q\in O(d)$ and all $\M\in\hmB(\IR^d,\IR^d)$,
  \begin{align*}
     \Theta_k(\vct Q C,\vct Q \M) & = \Theta_k(C,\M) .
  \end{align*}
  \item \emph{locality:} for all $\M\in\hmB(\IR^d,\IR^d)$,
  \begin{align}\label{eq:loc-supp-meas-sharp}
     \Theta_k(C,\M) & \leq \Theta_k(D,\M) , & \text{if } \;\; \mSkel_k(C)\cap\M & \subseteq \BL(D)\cap \M .
  \end{align}
        In particular,
  \begin{align*}
     \Theta_k(C,\M) & = \Theta_k(D,\M) , & \text{if } \;\; \mSkel_k(C)\cap\M & =\mSkel_k(D)\cap \M
  \\ & & \text{or } \;\; \BLs(C)\cap\M & = \BLs(D)\cap \M
  \\ & & \text{or } \;\; \BL(C)\cap\M & = \BL(D)\cap \M .
  \end{align*}
  \item \emph{product rule:} for all $\M,\N\in\hmB(\IR^d,\IR^d)$,
  \begin{align*}
     \Theta_k(C\times D,\M\bprod\N) & = \sum_{i+j=k} \Theta_i(C,\M)\;\Theta_j(D,\N) .
  \end{align*}
  \item \emph{polarity:} for all $\M\in\hmB(\IR^d,\IR^d)$,
  \begin{align}\label{eq:polarity-Theta_k}
     \Theta_k(C^\polar,\M) & = \Theta_{d-k}\big(C,\rev(\M)\big) .
  \end{align}
  \item \emph{additivity:} for all $\M\in\hmB(\IR^d)$,
  \begin{align*}
     \Theta_k(C+D,\M) + \Theta_k(C\cap D,\M) & = \Theta_k(C,\M) + \Theta_k(D,\M) , & \text{if } \;\; C+D & = C\cup D .
  \end{align*}
  \item \emph{weak continuity:} for all open sets $\mO\in\hmB(\IR^d,\IR^d)$, 
  \begin{align*}
     \liminf_i \; \Theta_k(C_i,\mO) & \geq \Theta_k(C,\mO) , & \text{if } \;\; \big\{C_i\mid i\in\IN\big\}\subseteq\P(\IR^d) \text{ such that } \lim_{i\to\infty}C_i & = C .
  \end{align*}
\end{enumerate}
\end{proposition}

\begin{proof}
Properties~(1) and~(2) follow, for example, from the formula~\eqref{eq:def-Theta-vect-Prob-Skel} for the support measures. For the locality property~\eqref{eq:loc-supp-meas-sharp} we will rely on a lemma that follows trivially from the Steiner formulas, which we discuss in Appendix~\ref{sec:Steiner-form}: according to Lemma~\ref{lem:loc-suppmeas}, $\Theta_k(C,\M)=\Theta_k(D,\M)$ if $\M\subseteq\BL(C)\cap\BL(D)$. Using this, we obtain from $\mSkel_k(C)\cap\M \subseteq \BL(D)\cap \M$,
\begin{align*}
   \Theta_k(C,\M) & = \Theta_k\big(C,\mSkel_k(C)\cap\M\big) \leq \Theta_k\big(C,\BL(D)\cap\M\big) = \Theta_k\big(C,\BL(D)\cap\M\cap\mSkel_k(C)\big)
\\ & \hspace{-5.3mm}\stackrel{\text{[Lem.~\ref{lem:loc-suppmeas}]}}{=} \Theta_k\big(D,\BL(D)\cap\M\cap\mSkel_k(C)\big) \leq \Theta_k\big(D,\BL(D)\cap\M\big) = \Theta_k(D,\M) ,
\end{align*}
which shows~\eqref{eq:loc-supp-meas-sharp}. The product rule follows for example from the formula~\eqref{eq:def-Theta-vect-Prob-no-proj} for the support measures and from the formula in~\eqref{eq:biconic-prod-idents} for the lifted $k$-skeleton $\mSkel_k(C\times D)$. The polarity formula follows from~\eqref{eq:def-Theta-vect-Prob-no-proj} and $\rev\big( \mSkel_k(C)\big) = \mSkel_{d-k}(C^\polar)$.

The easiest way to show the additivity and weak continuity of the support measures is to use the Steiner formulas, which we discuss in Appendix~\ref{sec:Steiner-form}. As we will use neither the additivity nor the weak continuity property, we will content ourselves with a description of the general idea in Appendix~\ref{sec:Steiner-form}; for more details we refer to~\cite[Sec.~6.5]{SW:08}.
\end{proof}

We will need the following lemma in the proof of the kinematic formula.

\begin{lemma}\label{lem:decomp-Theta_k(C,M)}
Let $C\in\P(\IR^d)$ and $\M\in\hmB(\IR^d,\IR^d)$ such that $\M\subseteq\BL(C)$. Then for every $0\leq k\leq d$,
  \[ \Theta_k(C,\M) = \sum_{L\in\mL_k(C)} \Theta_k\big(L,\M\cap \BL(L)\big) . \]
\end{lemma}

\begin{proof}
Since $\M\subseteq\BL(C)$ and $\BL(C)\subseteq\bigcup_{L\in\mL(C)}\BL(L)$, we have $\M=\bigcup_{L\in\mL(C)} (\M\cap \BL(L))$. Furthermore, for any two distinct linear subspaces $L\neq L'$, we have $\dim(\BL(L)\cap\BL(L'))<d$ so that $\Theta_k(C,\BL(L)\cap \BL(L'))=0$. Using the inclusion-exclusion principle, we thus obtain
\begin{align*}
   \Theta_k(C,\M) & = \Theta_k\bigg(C,\bigcup_{L\in\mL(C)}\big(\BL(L)\cap \M\big)\bigg) = \sum_{L\in\mL(C)} \Theta_k\big(C,\M\cap \BL(L)\big) .
\end{align*}
The locality property of the support measures implies $\Theta_k(C,\M\cap \BL(L)) = \Theta_k(L,\M\cap \BL(L))$, and of course $\Theta_k(L,\cdot)=0$ if $\dim L\neq k$, which finishes the proof.
\end{proof}

\subsection{Characterization}\label{sec:char-supp-meas}

The characterization theorem for the support measures enjoys a similar simplicity as the characterization for the polyhedral measures given in Theorem~\ref{thm:conic-char}. The locality assumption is this time expressed in terms of the biconic lift.

\begin{theorem}[Characterization of support measures]\label{thm:biconic-char}
Let $\psi\colon\P(\IR^d)\times\hmB(\IR^d,\IR^d)\to\IR$ be such that for all $C,D\in\P(\IR^d)$, $\M\in\hmB(\IR^d,\IR^d)$, $\vct Q\in O(d)$:
\begin{enumerate}\setcounter{enumi}{-1}
  \item $\psi(C,\cdot)\in\hmM(\IR^d,\IR^d)$,
  \item $\psi(C,\M)=\psi(C,\M\cap \BL(C))$,
  \item $\psi(\vct Q C,\vct Q \M)=\psi(C,\M)$,
  \item $\psi(C,\M)=\psi(D,\M)$ if $\BL(C)\cap \M=\BL(D)\cap \M$.
\end{enumerate}
Then $\psi$ is a linear combination of~$\mDir$ and $\Theta_0,\ldots,\Theta_d$. In this case
\begin{equation}\label{eq:exp-psi-ThetaII}
   \psi = \psi\big(\IR^d,\{\vct0\} \big)\;\mDir + \sum_{k=0}^d \psi\big(\Lk,\IR^{d+d}_*\big)\;\Theta_k ,
\end{equation}
where $\Lk\subseteq\IR^d$ denotes a $k$-dimensional subspace.
\end{theorem}

\begin{proof}
The proof follows broadly the same line of arguments as the proof of the characterization of the polyhedral measures given in Theorem~\ref{thm:conic-char}. An important tool is again the facial decomposition of~$C$, which yields two partitions of the biconic lift, cf.~\eqref{eq:decomp-Nor}.

Let $\psi\colon\P(\IR^d)\times\hmB(\IR^d,\IR^d)\to\IR$ satisfy the assumptions (0)--(3), i.e., $\psi(C,\cdot)$ is a biconic measure, which is concentrated on~$\BL(C)$, $\psi$~is invariant under (simultaneous) orthogonal transformations, and $\psi(C,\M)=\psi(D,\M)$, if $\BL(C)\cap\M=\BL(D)\cap\M$.

Again, we first consider the case of a linear subspace $L\subseteq\IR^d$. For $M\in\hmB(L)$ define
  \[ \psi_{L,M} \colon \hmB(L^\bot)\to\IR ,\qquad \psi_{L,M}(M') := \psi(L,M\times M') . \]
Orthogonal invariance of~$\psi$ implies that~$\psi_{L,M}$ is an orthogonal invariant conic measure on~$L^\bot$, so that by Lemma~\ref{lem:orth-inv-conic-meas},
\begin{equation}\label{eq:psi(L,MxM')=...-1}
  \psi(L,M\times M') = \psi(L,M\times \{\vct0\})\,\Dir(M') + \psi(L,M\times L^\bot_*)\,\gamma_{L^\bot}(M')
\end{equation}
for every $M'\in\hmB(L^\bot)$. Consider now the measures
\begin{align*}
   \psi_L^{(0)},\psi_L^{(1)} & \colon\hmB(L) \to\IR , & \psi_L^{(0)}(M) & := \psi(L,M\times \{\vct0\}) , & \psi_L^{(1)}(M) & := \psi(L,M\times L^\bot_*) .
\end{align*}
By the same reasoning as above we arrive at
\begin{align*}
   \psi(L,M\times \{\vct0\}) & = \psi(L,\{\vct0\}\times \{\vct0\})\,\Dir(M) + \psi(L,L_*\times \{\vct0\})\,\gamma_L(M) ,
\\ \psi(L,M\times L^\bot_*) & = \psi(L,\{\vct0\}\times L^\bot_*)\,\Dir(M) + \psi(L,L_*\times L^\bot_*)\,\gamma_L(M) .
\end{align*}
Orthogonal invariance of~$\psi$ implies that the constants in the above expressions only depend on~$\dim(L)=:k$, which we denote for simplicity by (using the concentration property~(1) of $\psi$)
\begin{align*}
   r_k & := \psi(L,\{\vct0\}\times \{\vct0\}) , & s_k & := \psi(L,\IR^d_*\times \{\vct0\}) , & t_k & := \psi(L,\{\vct0\}\times \IR^d_*) , & u_k & := \psi(L,\IR^d_*\times \IR^d_*) .
\end{align*}
In fact, the locality property~(3) of~$\psi$ implies
  \[ r_0 = r_1 = \cdots = r_d = \psi(\IR^d,\{\vct0\}\times \{\vct0\}) =: r . \]
We thus obtain from~\eqref{eq:psi(L,MxM')=...-1}, for $M\in\hmB(L)$, $M'\in\hmB(L^\bot)$,
\begin{align}\label{eq:psi(L,MxM')=...-2}
   \psi( L, M\times M') & = r\,\Dir(M)\,\Dir(M') + s_k\,\gamma_L(M)\,\Dir(M') + t_k\,\Dir(M)\,\gamma_{L^\bot}(M') + u_k\,\gamma_L(M)\,\gamma_{L^\bot}(M') .
\end{align}
Note that $s_0=t_d=u_0=u_d=0$.

Consider now a general polyhedral cone $C\in\P(\IR^d)$. The biconic lift of~$C$ has two disjoint decompositions, cf.~\eqref{eq:decomp-Nor}, which we recall here for convenience:
\begin{align}\label{eq:decomp-Nor-recall}
  \BL(C) & = \Disjunionu{L\in\mL(C)} \big(\inter_L(C\cap L)\times (C^\polar\cap L^\bot)\big) = \Disjunionu{L\in\mL(C)} \big((C\cap L)\times \inter_{L^\bot}(C^\polar\cap L^\bot)\big) .
\end{align}
Using the assumptions (0)--(3) and the first decomposition of $\BL(C)$ in~\eqref{eq:decomp-Nor-recall}, we obtain for $M,M'\in\hmB(\IR^d)$
\begin{align*}
   & \psi(C,M\times M') \stackrel{(1)}{=} \psi(C,\BL(C)\cap (M\times M')) \stackrel{\eqref{eq:decomp-Nor-recall}}{=} \sum_{L\in\mL(C)} \psi\big(C,(\inter_L(C\cap L)\cap M)\times (C^\polar\cap L^\bot\cap M')\big)
\\ & \stackrel{(3)}{=} \sum_{L\in\mL(C)} \psi\big(L,(\inter_L(C\cap L)\cap M)\times (C^\polar\cap L^\bot\cap M')\big)
\\ & \stackrel{\eqref{eq:psi(L,MxM')=...-2}}{=} r\,\Dir(M') \sum_{L\in\mL(C)} \Dir(\inter_L(C\cap L)\cap M) + \Dir(M')\sum_{k=1}^d s_k \sum_{L\in\mL_k(C)} \gamma_L(C\cap L\cap M)
\\ & + \sum_{k=0}^{d-1} t_k \!\!\sum_{L\in\mL_k(C)}\!\! \Dir(\inter_L(C\cap L)\cap M)\,\gamma_{L^\bot}(C^\polar\cap L^\bot\cap M') + \sum_{k=1}^{d-1} u_k \!\!\sum_{L\in\mL_k(C)}\!\! \gamma_L(C\cap L\cap M)\,\gamma_{L^\bot}(C^\polar\cap L^\bot\cap M') .
\end{align*}
Note that the origin lies in the relative interior of a unique face of~$C$ of dimension~$\lin(C)$. Using this, we obtain
\begin{align*}
   \sum_{L\in\mL(C)} \Dir(\inter_L(C\cap L)\cap M)\, \Dir(M') & = \Dir(M)\,\Dir(M') = \mDir(M\times M') \quad \text{ and}
\\ \sum_{k=0}^d t_k\sum_{L\in\mL_k(C)} \Dir(\inter_L(C\cap L)\cap M)\,\gamma_{L^\bot}(C^\polar\cap L^\bot\cap M') & = t_{\lin(C)} \Dir(M)\sum_{L\in\mL_{\lin(C)}(C)}\gamma_{L^\bot}(C^\polar\cap L^\bot\cap M') .
\end{align*}
Furthermore, note that
\begin{align*}
   \sum_{L\in\mL_k(C)} \gamma_L(C\cap L\cap M) & =\Psi_k(C,M) , & \sum_{L\in\mL_k(C)} \gamma_{L^\bot}(C^\polar\cap L^\bot\cap M') & =\Psi_{d-k}(C^\polar,M') ,
\\ & \hspace{-33.3mm}\sum_{L\in\mL_k(C)} \gamma_L(C\cap L\cap M)\,\gamma_{L^\bot}(C^\polar\cap L^\bot\cap M') = \Theta_k(C,M\times M') . \hspace{-5cm}
\end{align*}
Using Lemma~\ref{lem:coinc-dir-prod}, we arrive at
\begin{equation}\label{eq:psi=lin-comb-1}
   \psi(C,\cdot) = r\mDir + \sum_{k=1}^d s_k \Psi_k(C,\cdot)\otimes \Dir + \sum_{k=0}^{d-1} t_k \Lin_k(C,\cdot)\otimes \Psi_{d-k}(C^\polar,\cdot) + \sum_{k=1}^{d-1} u_k \Theta_k(C,\cdot) .
\end{equation}
It remains to show that $s_1=\cdots=s_{d-1}=t_1=\cdots=t_{d-1}=0$. This is achieved by using the second decomposition of $\BL(C)$ in~\eqref{eq:decomp-Nor-recall}. Analogously to the above,
\begin{align*}
   \psi(C,M\times M') & = \sum_{k=0}^d \; \sum_{L\in\mL_k(C)} \psi\big(L,(C\cap L\cap M)\times (\inter_{L^\bot}(C^\polar\cap L^\bot)\cap M')\big) ,
\end{align*}
and from this it follows
\begin{equation}\label{eq:psi=lin-comb-2}
   \psi(C,\cdot) = r\mDir + \sum_{k=1}^d s_k \Psi_k(C,\cdot)\otimes \Lin_{d-k}(C^\polar,\cdot) + \sum_{k=0}^{d-1} t_k  \Dir\otimes \Psi_{d-k}(C^\polar,\cdot) + \sum_{k=1}^{d-1} u_k \Theta_k(C,\cdot) .
\end{equation}

Now, let $C\in\P(\IR^d)$ such that $\inter(C)\neq\emptyset$ (equivalently, $\lin(C^\polar)=0$). Then for every $L\in\mL_m(C)$ with $1\leq m\leq d-1$, we have $\Psi_k(C,L)=0$ if $k>m$, and overall
\begin{align*}
   \psi(C,L_*\times\{\vct0\}) & = \begin{cases} \sum_{k=1}^m s_k \Psi_k(C,L) & \text{by~\eqref{eq:psi=lin-comb-1}} \\ 0 & \text{by~\eqref{eq:psi=lin-comb-2}} . \end{cases}
\end{align*}
Assuming~$C$ is such that $\inter(C)\neq\emptyset$ and $\mL_m(C)\neq\emptyset$ for $m=1,\ldots,d-1$, we obtain $s_1=\cdots=s_{d-1}=0$. Analogously, by considering $\psi(C,\{\vct0\}\times L_*)$ we obtain $t_1=\cdots=t_{d-1}=0$. Finally, note that
\begin{align*}
   \Psi_d(C,\cdot)\otimes \Dir & = \Psi_d(C,\cdot)\otimes \Lin_0(C^\polar,\cdot) = \Theta_d(C,\cdot) , & \Lin_0(C,\cdot)\otimes \Psi_d(C^\polar,\cdot) & = \Dir\otimes \Psi_d(C^\polar,\cdot) = \Theta_0(C,\cdot) ,
\end{align*}
so that we obtain
\begin{align*}
   \psi & = r\mDir + s_d\Theta_d + t_0\Theta_0 + \sum_{k=1}^{d-1} u_k \Theta_k = \psi(\IR^d,\{\vct0\}\times\{\vct0\})\mDir + \sum_{k=0}^d \psi(\Lk,\IR^d_*\times\IR^d_*)\,\Theta_k . \qedhere
\end{align*}
\end{proof}

\subsection{Kinematic formulas}\label{sec:kinform-suppmeas}

In this section we will use the characterization of the support measures provided by Theorem~\ref{thm:biconic-char} to prove the following theorem.

\begin{theorem}[Kinematic formula for support measures]\label{thm:kinform-supp}
Let $C_0,\ldots,C_n\in\P(\IR^d)$, $\M_0,\ldots,\M_n\in\hmB(\IR^d)$ such that $\M_i\subseteq\BL(C_i)$ for all~$0\leq i\leq n$, and let $\vct T_0,\ldots,\vct T_n\in\Gl_d$ and $k>0$. Then for $\vct Q_0,\ldots,\vct Q_n\in O(d)$ iid uniformly at random,
\begin{equation}\label{eq:kinform-supp}
   \Expect\bigg[ \Theta_k\bigg( \bigcap_{i=0}^n \vct T_i \vct Q_i C_i , \bigwedge_{i=0}^n \vct T_i \vct Q_i \M_i \bigg)\bigg] = \Theta_{nd+k}(C,\M) ,
\end{equation}
where $C:=C_0\times\cdots\times C_n$, $\M:=\M_0\bprod \cdots\bprod \M_n$. The support measures of the product are given by
\begin{align}\label{eq:supp-meas-n}
  \Theta_m(C,\M) & = \sum_{i_0+\cdots+i_n=m} \Theta_{i_0}(C_0,\M_0)\cdots \Theta_{i_n}(C_n,\M_n) .
\end{align}
\end{theorem}

Again, the product formula~\eqref{eq:supp-meas-n} is just an iteration of~(4) in Proposition~\ref{prop:props-supp-meas}. The polarity property of the support measures~\eqref{eq:polarity-Theta_k} immediately implies the following corollary from Theorem~\ref{thm:kinform-supp}.

\begin{corollary}[Polar kinematic formula for support measures]\label{cor:polar-kinform-supp}
Let the notation and assumptions be as in Theorem~\ref{thm:kinform-supp}. Then
\begin{equation}\label{eq:kinform-supp-polar}
   \Expect\bigg[ \Theta_{d-k}\bigg( \sum_{i=0}^n \vct T_i\vct Q_i C_i , \bigvee_{i=0}^n \vct T_i\vct Q_i \M_i\bigg)\bigg] = \Theta_{d-k}(C,\M) .
\end{equation}
\end{corollary}

\begin{proof}
We have $(\vct T_i\vct Q_i C_i)^\polar=\vct T_i^\invadj\vct Q_i C_i^\polar$ and $\rev(\vct T_i\vct Q_i \M_i)=\vct T_i^\invadj\vct Q_i \rev(\M_i)$, as well as
\begin{align*}
   C_0^\polar\times\cdots\times C_n^\polar & = (C_0\times\cdots\times C_n)^\polar = C^\polar , & \rev(\M_0)\bprod\cdots\bprod\rev(\M_n) & = \rev(\M) ,
\end{align*}
where it should be noted that $C_i^\polar$ denotes the polar in~$\IR^d$, whereas $C^\polar$ denotes the polar in~$\IR^{(n+1)d}$, and similarly for $\rev(\M_i)$ and $\rev(\M)$. Using this, we obtain
\begin{align*}
   \Expect\bigg[ \Theta_{d-k}\bigg( \sum_{i=0}^n \vct T_i\vct Q_i C_i , \bigvee_{i=0}^n \vct T_i\vct Q_i \M_i\bigg)\bigg] & \stackrel{\eqref{eq:De-Morg-biconic},\eqref{eq:polarity-Theta_k}}{=} \Expect\bigg[ \Theta_k\bigg( \bigcap_{i=0}^n \vct T_i^\invadj\vct Q_i C_i^\polar , \bigwedge_{i=0}^n \vct T_i^\invadj\vct Q_i \rev(\M_i)\bigg)\bigg]
\\ & \stackrel{\eqref{eq:kinform-supp}}{=} \Theta_{nd+k}\big(C^\polar,\rev(\M)\big) = \Theta_{d-k}(C,\M) . \qedhere
\end{align*}
\end{proof}

\begin{remark}\label{rem:proj-form-curvmeas}
A special case of~\eqref{eq:kinform-supp-polar} is the dual projection formula for the curvature measures: if $L\subseteq\IR^d$ denotes a uniformly random linear subspace of codimension~$m$, say, $L=\vct QL_0$ for $\vct Q\in O(d)$ uniformly at random, and if $M\in\hmB(\IR^d)$ with $M\subseteq C$, then for $0\leq k\leq d-m-1$
\begin{align*}
   \Expect\big[ \Phi_k(\Pi_L(C),\Pi_L(M)) \big] & = \Expect\big[ \Theta_{k+m}\big((C,M\times C^\polar)\vee \vct Q(L_0^\bot,L_0^\bot\times L_0)\big) \big]
\\ & \stackrel{\eqref{eq:kinform-supp-polar}}{=} \sum_{i+j=k+m} \Theta_i(C,M\times C^\polar)\,\Theta_j(L_0^\bot,L_0^\bot\times L_0) = \Phi_k(C,M) .
\end{align*}
This formula is useful in the probabilistic analysis of convex programming, see~\cite[Thm.~2.1]{ambu:12}.
\end{remark}

The general argumentation in the proof of Theorem~\ref{thm:kinform-supp} will be as in Section~\ref{sec:calc-kinem-conic}. However, the biconic $\wedge$-operation requires extra care, cf.~Proposition~\ref{prop:biconic-conj}, and some important steps rely on genericity arguments, which we again defer to Appendix~\ref{sec:genericity} to ease the presentation. For example, see Proposition~\ref{prop:genericity}(1) for the well-definedness of the left-hand side in~\eqref{eq:kinform-supp}.

The formulation of Theorem~\ref{thm:kinform-supp} is such that one assumes $\M_i\subseteq\BL(C_i)$, $0\leq i\leq n$. One can drop this assumption if instead of~\eqref{eq:kinform-supp} one considers the following claim:
\begin{equation}\label{eq:kinform-supp-BL(C)}
   \Expect\bigg[ \Theta_k\bigg( \bigcap_{i=0}^n \vct T_i \vct Q_i C_i , \bigwedge_{i=0}^n \vct T_i \vct Q_i \big(\BL(C_i)\cap\M_i\big)\bigg)\bigg] = \Theta_{nd+k}(C,\M) .
\end{equation}
We will use this version of Theorem~\ref{thm:kinform-supp} in the rest of this section.

\begin{lemma}\label{lem:Theta-satisf-assumpts-biconic}
Let $C_1,\ldots,C_n\in\P(\IR^d)$, $\M_1,\ldots,\M_n\in\hmB(\IR^d,\IR^d)$, $\vct T_0,\ldots,\vct T_n\in\Gl_d$, and $k>0$. Then for $\vct Q_0,\ldots,\vct Q_n\in O(d)$ iid uniformly at random, the map $\psi\colon\P(\IR^d)\times\hmB(\IR^d,\IR^d)\to\IR$ given by
\begin{equation}\label{eq:kinform-supp_lem}
   \psi(C_0,\M_0) = \Expect\bigg[ \Theta_k\bigg( \bigcap_{i=0}^n \vct T_i \vct Q_i C_i , \bigwedge_{i=0}^n \vct T_i \vct Q_i \big(\BL(C_i)\cap\M_i\big)\bigg)\bigg]
\end{equation}
satisfies the assumptions in Theorem~\ref{thm:biconic-char}.
\end{lemma}

\begin{proof}
To simplify the notation we abbreviate $\vct T_i\vct Q_i =: \vct U_i$.

(0) \emph{Claim:} $\psi(C_0,\cdot)\in\hmM(\IR^d,\IR^d)$. Let $\N_j\in\hmB(\IR^d,\IR^d)$, $j=1,2,\ldots$, be pairwise disjoint. Clearly, $\vct U_0\big(\BL(C_0)\cap \N_j\big)$, $j=1,2,\ldots$, are pairwise disjoint as well, so we may assume that $\N_j\subseteq \BL(C_0)$ for all~$j$. It is not true that $\vct U_0\N_j \wedge \M$, $j=1,2,\ldots$, with $\M:=\bigwedge_{i=1}^n \vct U_i \big( \BL(C_i)\cap\M_i\big)$, are pairwise disjoint, cf.~Proposition~\ref{prop:biconic-conj}. However, denoting $D:=\bigcap_{i=1}^n \vct U_i C_i$, we have
  \[ \M = \bigwedge_{i=1}^n \vct U_i \big( \BL(C_i)\cap\M_i\big) \subseteq \bigwedge_{i=1}^n \vct U_i \BL(C_i) = \BL\bigg(\bigcap_{i=1}^n \vct U_i C_i\bigg) = \BL(D) , \]
so that Proposition~\ref{prop:genericity_BL} in Appendix~\ref{sec:genericity} implies
\begin{align*}
   \psi\bigg(C_0,\bigcup_{j=1}^\infty \N_j\bigg) & = \Expect\bigg[ \Theta_k\bigg( \vct U_0 C_0\cap D , \bigg(\bigcup_{j=1}^\infty \vct U_0\N_j\bigg) \wedge \M \bigg)\bigg] = \Expect\bigg[ \sum_{j=1}^\infty \Theta_k\big( \vct U_0 C_0\cap D , \vct U_0\N_j \wedge \M\big)\bigg]
\\ & \stackrel{(*)}{=} \sum_{j=1}^\infty \Expect\big[ \Theta_k\big( \vct U_0 C_0\cap D , \vct U_0\N_j \wedge \M\big)\big] = \sum_{j=1}^\infty \psi(C_0,\N_j) ,
\end{align*}
where $(*)$ follows from an application of the monotone convergence theorem.

(1) \emph{Claim:} $\psi(C_0,\M_0)=\psi\big(C_0,\M_0\cap \BL(C_0)\big)$. This follows directly from the definition of~$\psi$.

(2) \emph{Claim:} $\psi(\vct Q C_0,\vct Q_0 \M)=\psi(C_0,\M_0)$ for $\vct Q\in O(d)$. This also follows directly from the definition of~$\psi$, since $\vct Q_0\vct Q$ is uniformly at random in~$O(d)$ and independent of $\vct Q_1,\ldots,\vct Q_n$.

(3) \emph{Claim:} $\psi(C_0,\M_0)=\psi(\tilde C_0,\M_0)$ if $\BL(C_0)\cap \M_0=\BL(\tilde C_0)\cap \M_0$. Define $\tilde C_i:=C_i$ for $i=1,\ldots,n$, and $C := \bigcap_{i=0}^n \vct U_i C_i$, $\tilde C := \bigcap_{i=0}^n \vct U_i \tilde C_i$, $\M := \bigwedge_{i=0}^n \vct U_i\big( \BL(C_i)\cap \M_i\big)$. Then we have
\begin{align*}
   \BL(C)\cap\M & = \bigg(\bigwedge_{i=0}^n \vct U_i \BL(C_i)\bigg) \cap \bigg(\bigwedge_{i=0}^n \big( \vct U_i\BL(C_i)\cap \vct U_i\M_i\big)\bigg) \stackrel{(\dagger)}{=} \bigwedge_{i=0}^n \big( \vct U_i\BL(C_i)\cap \vct U_i\M_i\big)
\\ & \stackrel{(\ddagger)}{=} \bigwedge_{i=0}^n \big( \vct U_i\BL(\tilde C_i)\cap \vct U_i\M_i\big) \stackrel{(\dagger)}{=} \bigg(\bigwedge_{i=0}^n \vct U_i \BL(\tilde C_i)\bigg) \cap \bigg(\bigwedge_{i=0}^n \big( \vct U_i\BL(\tilde C_i)\cap \vct U_i\M_i\big)\bigg) = \BL(\tilde C)\cap\M ,
\end{align*}
where $(\dagger)$ follows from Proposition~\ref{prop:biconic-conj} and $(\ddagger)$ follows from the assumption $\BL(C_0)\cap \M_0=\BL(\tilde C_0)\cap \M_0$. By the locality of the support measures, cf.~Proposition~\ref{prop:props-supp-meas}(3), it follows that $\Theta_k(C,\M)=\Theta_k(\tilde{C},\M)$; in particular the expectations coincide.
\end{proof}

\begin{lemma}\label{lem:prod-exch-biconic}
Let $k>0$, $C,D\in\P(\IR^d)$, $\M,\N\in\hmB(\IR^d,\IR^d)$. Then
  \[ \sum_{j=1}^d \Theta_j(C,\M) \, \Theta_{d+k}(L_j\times D,\IR^{d+d}\bprod \N) = \Theta_{d+k}(C\times D,\M\bprod \N) , \]
where $L_j\subseteq\IR^d$ a $j$-dimensional linear subspace.
\end{lemma}

\begin{proof}
Argue exactly as in Lemma~\ref{lem:prod-exch}.
\end{proof}

\begin{proof}[Proof of Theorem~\ref{thm:kinform-supp}]
To simplify the notation we abbreviate $\vct T_i\vct Q_i =: \vct U_i$. Define
\begin{align*}
   \LEFT(C_0,\ldots,C_n;\M_0,\ldots,\M_n) & := \Expect\bigg[ \Theta_k\bigg( \bigcap_{i=0}^n \vct U_i C_i , \bigwedge_{i=0}^n \vct U_i \big(\BL(C_i)\cap\M_i\big)\bigg)\bigg] ,
\\ \RIGHT(C_0,\ldots,C_n;\M_0,\ldots,\M_n) & := \Theta_{nd+k}\big( C_0\times\cdots\times C_n ,\, \M_0\bprod \cdots\bprod \M_n \big) ,
\end{align*}
so that we need to show $\LEFT(C_0,\ldots,C_n;\M_0,\ldots,\M_n) = \RIGHT(C_0,\ldots,C_n;\M_0,\ldots,\M_n)$. By induction on~$m$ we will show that
\begin{align*}
   \LEFT\big(C_0,\ldots,C_{m-1},L^{(0)},\ldots,L^{(n-m)};\M_0,\ldots,\M_{m-1},\IR^{d+d},\ldots,\IR^{d+d}\big) &
\\ = \RIGHT\big(C_0,\ldots,C_{m-1},L^{(0)},\ldots,L^{(n-m)};\M_0,\ldots,\M_{m-1},\IR^{d+d},\ldots,\IR^{d+d}\big) &
\end{align*}
where $L^{(0)},\ldots,L^{(n-m)}\subseteq\IR^d$ linear subspaces.

The case $m=0$ is easily established: the intersection $\bigcap_{i=0}^n \vct U_i L^{(i)}$ is almost surely a linear subspace of dimension $\max\{\dim L^{(0)}+\cdots+\dim L^{(n)}-nd,0\}$, and the direct product $L^{(0)}\times\cdots\times L^{(n)}$ is a linear subspace of dimension $\dim L^{(0)}+\cdots+\dim L^{(n)}$, so that
\begin{align*}
   \LEFT\big(L^{(0)},\ldots,L^{(n)};\IR^{d+d},\ldots,\IR^{d+d}\big) & = \begin{cases} 1 & \text{if } nd+k = \dim L^{(0)}+\cdots+\dim L^{(n)} \\ 0 & \text{else} \end{cases}
\\ & = \RIGHT\big(L^{(0)},\ldots,L^{(n)};\IR^{d+d},\ldots,\IR^{d+d}\big) .
\end{align*}

For the induction step, $m\geq1$, we define $\psi\colon\P(\IR^d)\times\hmB(\IR^d,\IR^d)\to\IR$,
\begin{align*}
   \psi(C,\M) & := \LEFT\big(C_0,\ldots,C_{m-2},C,L^{(0)},\ldots,L^{(n-m)};\M_0,\ldots,\M_{m-2},\M,\IR^{d+d},\ldots,\IR^{d+d})
\\ & = \Expect\bigg[ \Theta_k\bigg( \bigcap_{i=0}^{m-2} \vct U_i C_i \cap \vct U C\cap \bigcap_{i=0}^{n-m} \vct U_{m+i} L^{(i)} , \bigwedge_{i=0}^{m-2} \vct U_i \M_i \wedge \vct U \M \bigg)\bigg] ,
\end{align*}
where we set $\vct U:=\vct U_{m-1}$ to simplify the notation. Note that $\Psi_k(C,\{\vct0\})=0$, since $k>0$, and thus $\psi(C,\{\vct0\})=0$. By Lemma~\ref{lem:Theta-satisf-assumpts-biconic}, $\psi$ satisfies the assumptions in Theorem~\ref{thm:biconic-char}, so that
\begin{align*}
   & \LEFT\big(C_0,\ldots,C_{m-1},L^{(0)},\ldots,L^{(n-m)};\M_0,\ldots,\M_{m-1},\IR^{d+d},\ldots,\IR^{d+d}\big)
\\ & = \psi(C_{m-1},\M_{m-1}) \stackrel{\eqref{eq:exp-psi-ThetaII}}{=} \underbrace{\psi(\IR^d,\{\vct0\})}_{=0}\,\mDir(\M_{m-1}) + \sum_{j=1}^d \underbrace{\psi(L_j,\IR^{d+d}_*)}_{=\psi(L_j,\IR^{d+d})}\,\Theta_j(C_{m-1},\M_{m-1})
\\ & = \sum_{j=1}^d \LEFT\big(C_0,\ldots,C_{m-2},L_j,L^{(0)},\ldots,L^{(n-m)};\M_0,\ldots,\M_{m-2},\IR^{d+d},\ldots,\IR^{d+d}\big)\,\Theta_j(C_{m-1},\M_{m-1})
\\ & \stackrel{\text{(IH)}}{=} \sum_{j=1}^d \Theta_{nd+k}\big( C_0\times\cdots\times C_{m-2}\times L_j\times L^{(0)}\times \cdots\times L^{(n-m)} ,\, \M_0\bprod \cdots\bprod \M_{m-2}\bprod \IR^{d+d}\bprod \cdots\bprod \IR^{d+d}\big)
\\ & \hspace{15mm} \cdot \Theta_j(C_{m-1},\M_{m-1})
\\ & \stackrel{\text{[Lem.~\ref{lem:prod-exch-biconic}]}}{=} \RIGHT\big(C_0,\ldots,C_{m-1},L^{(0)},\ldots,L^{(n-m)};M_0,\ldots,M_{m-1},\IR^d,\ldots,\IR^d\big) ,
\end{align*}
which shows the induction step, and thus finishes the proof.
\end{proof}

\section{General kinematic formulas}\label{sec:gen-kin-form}

In this section we provide the proof for the general kinematic formula stated in Theorem~\ref{thm:gen-kinform-v} (with the restriction $\vct T_0,\ldots,\vct T_n\in O(d)$). Since the proof is a simple induction on the number of indeterminates in the Boolean formula, essentially the same proof yields a general kinematic formula for the support measures, which we state next. We will only prove Theorem~\ref{thm:gen-kinform-v}, as the proof translates straightforwardly to the support measures case.

If $F(X_0,\ldots,X_n)$ denotes a Boolean formula in the variables $X_0,\ldots,X_n$, and if $\M_0,\ldots,\M_n\in\hmB(\IR^d,\IR^d)$, we define the evaluation $F(\M_0,\ldots,\M_n)\in\hmB(\IR^d,\IR^d)$ to be the result of replacing negation~$\neg(\cdots)$ by the reversal map~$\rev(\cdots)$.

\begin{theorem}[General kinematic formula for support measures]\label{conj:gen-kinform-supp}
Let $C_0,\ldots,C_n\in\P(\IR^d)$, $\M_0,\ldots,\M_n\in\hmB(\IR^d)$ such that $\M_i\subseteq\BL(C_i)$ for all~$0\leq i\leq n$, and let $\vct T_0,\ldots,\vct T_n\in O(d)$ and $0< k< d$. Furthermore, let $F(X_0,\ldots,X_n)$ be a Boolean read-once formula. Then for $\vct Q_0,\ldots,\vct Q_n\in O(d)$ iid uniformly at random,
\begin{align}
  \Expect\bigg[ \Theta_k\Big( F\big(\vct T_0\vct Q_0 C_0,\ldots,\vct T_n\vct Q_n C_n\big) , F\big( & \vct T_0\vct Q_0 \M_0,\ldots,\vct T_n\vct Q_n \M_n\big) \Big)\bigg]
\nonumber
\\ & = \sum_{\dim^F_d(k_0,\ldots,k_n)=k} \Theta_{k_0}(C_0,\M_0) \cdots \Theta_{k_n}(C_n,\M_n) .
\label{eq:gen-kinform-supp}
\end{align}
\end{theorem}

As remarked for Theorem~\ref{thm:gen-kinform-v}, the transformations $\vct T_0,\ldots,\vct T_n$ in~\eqref{eq:gen-kinform-supp} can of course be dropped entirely; we included them to simplify the comparison to the other formulas. It seems reasonable to assume that~\eqref{eq:gen-kinform-supp} also holds for general $\vct T_0,\ldots,\vct T_n\in\Gl_d$. A possible approach to this would be to merge the reasoning for the polar kinematic formula given in Corollary~\ref{cor:polar-kinform-supp} with the induction step in the proof of Theorem~\ref{thm:kinform-supp}. However, it is by no means clear that this proof idea can be implemented successfully, due to the subtleties involving the biconic $\wedge$- and $\vee$-operations.

The boundary cases for the intrinsic volumes~\eqref{eq:kinform-v-bd_cases} do not localize to the support measures. Simple counter-examples can be found already in dimension one or two, which show that in general (notation as in Theorem~\ref{thm:kinform-supp})
\begin{align*}
  \Expect\bigg[ \Theta_0\bigg( \bigcap_{i=0}^n \vct T_i\vct Q_i C_i , \bigwedge_{i=0}^n \vct T_i\vct Q_i \M_i\bigg)\bigg] & \!\neq\! \sum_{j=0}^{nd} \Theta_j(C,\M) , & \Expect\bigg[ \Theta_d\bigg( \sum_{i=0}^n \vct T_i\vct Q_i C_i , \bigvee_{i=0}^n \vct T_i\vct Q_i \M_i\bigg)\bigg] & \!\neq\! \sum_{j=0}^{nd} \Theta_{d+j}(C,\M) .
\end{align*}

We finish this section with the proof of the general kinematic formula~\eqref{eq:gen-kinform-v} for the intrinsic volumes, where we assume that $\vct T_0,\ldots,\vct T_n\in O(d)$, so that we may as well drop these transformations entirely in~\eqref{eq:gen-kinform-v}.

\begin{proof}[Proof of Theorem~\ref{thm:gen-kinform-v}]
Note first that since the lattice of polyhedral cones satisfies the De Morgan's Laws $(C\cap D)^\polar = C^\polar + D^\polar$ and $(C + D)^\polar = C^\polar \cap D^\polar$, we may assume without loss of generality that all negations in $F(X_0,\ldots,X_n)$ are directly at the variables. We now proceed by induction on~$n$.

If $n=0$ then $F(X_0)=X_0$ or $F(X_0)=\neg X_0$. In this case we have $\dim^F_d(k_0)=k_0$ or $\dim^F_d(k_0)=d-k_0$, respectively. Therefore, we have
\begin{align*}
   \Expect\Big[ v_k\big( F(\vct Q_0 C_0)\big)\Big] & = v_k(C_0) = \sum_{\dim^F_d(k_0)=k} v_{k_0}(C_0) ,\qquad\text{or}
\\ \Expect\Big[ v_k\big( F(\vct Q_0 C_0)\big)\Big] & = v_k(C_0^\polar) = v_{d-k}(C_0) = \sum_{\dim^F_d(k_0)=k} v_{k_0}(C_0) ,
\end{align*}
respectively. This settles the case $n=0$.

For $n>0$ we can permutate the variables $X_0,\ldots,X_n$ in such a way that (without loss of generality) we have for some $0\leq m<n$
\begin{align*}
   F(X_0,\ldots,X_n) & = F_1(X_0,\ldots,X_m)\wedge F_2(X_{m+1},\ldots,X_n) ,\quad\text{or}
\\ F(X_0,\ldots,X_n) & = F_1(X_0,\ldots,X_m)\vee F_2(X_{m+1},\ldots,X_n) .
\end{align*}
In the first case we have
  \[ \dim^F_d(k_0,\ldots,k_n) = \max\Big\{0,\dim^{F_1}_d(k_0,\ldots,k_m) + \dim^{F_2}_d(k_{m+1},\ldots,k_n)-d\Big\} , \]
and we may argue in the following way:
\begin{align}
   \underset{\vct Q_0,\ldots,\vct Q_n}\Expect\bigg[ & v_k\Big( F\big(\vct Q_0 C_0,\ldots,\vct Q_n C_n\big) \Big)\bigg] = \underset{\vct Q_0,\ldots,\vct Q_n}\Expect\bigg[ v_k\Big( F_1\big(\vct Q_0 C_0,\ldots,\vct Q_m C_m\big) \wedge F_2\big(\vct Q_{m+1} C_{m+1},\ldots,\vct Q_n C_n\big) \Big)\bigg]
\nonumber
\\ & \quad = \underset{\vct Q_0,\ldots,\vct Q_n}\Expect\Bigg[ \underset{\vct Q,\vct Q'}\Expect\bigg[ v_k\Big( \vct Q F_1\big(\vct Q_0 C_0,\ldots,\vct Q_m C_m\big) \wedge \vct Q'F_2\big(\vct Q_{m+1} C_{m+1},\ldots,\vct Q_n C_n\big) \Big)\bigg]\Bigg] ,
\label{eq:gen-kinform-v_ind1}
\end{align}
where in the second step we have replaced $\vct Q_0,\ldots,\vct Q_m$ and $\vct Q_{m+1},\ldots,\vct Q_n$ by $\vct{QQ}_0,\ldots,\vct{QQ}_m$ and $\vct Q'\vct Q_{m+1},\ldots,\vct Q'\vct Q_n$, respectively, with $\vct Q,\vct Q'\in O(d)$ iid uniformly at random.\footnote{This is the place where we make explicit use of the fact that the linear transformations $\vct T_0,\ldots,\vct T_n$ in~\eqref{eq:gen-kinform-v} are in $O(d)$.} Applying the (normal) kinematic formula~\eqref{eq:kinform-v} to the inner expectation yields
\begin{align*}
   \eqref{eq:gen-kinform-v_ind1} & = \underset{\vct Q_0,\ldots,\vct Q_n}\Expect\Bigg[ \sum_{\max\{i+j-d\}=k} v_i\Big( F_1\big(\vct Q_0 C_0,\ldots,\vct Q_m C_m\big)\Big) \, v_j\Big(F_2\big(\vct Q_{m+1} C_{m+1},\ldots,\vct Q_n C_n\big) \Big)\Bigg]
\\ & \quad = \sum_{\max\{i+j-d\}=k} \; \underset{\vct Q_0,\ldots,\vct Q_m}\Expect\bigg[ v_i\Big( F_1\big(\vct Q_0 C_0,\ldots,\vct Q_m C_m\big)\Big)\bigg] \, \underset{\vct Q_{m+1},\ldots,\vct Q_n}\Expect\bigg[ v_j\Big(F_2\big(\vct Q_{m+1} C_{m+1},\ldots,\vct Q_n C_n\big) \Big)\bigg]
\\ & \quad \stackrel{(*)}{=} \sum_{\max\{i+j-d\}=k} \; \sum_{\dim^{F_1}_d(k_0,\ldots,k_m)=i} v_{k_0}(C_0) \cdots v_{k_m}(C_m) \, \sum_{\dim^{F_2}_d(k_{m+1},\ldots,k_n)=j} v_{k_{m+1}}(C_{m+1}) \cdots v_{k_n}(C_n)
\\ & \quad = \sum_{\dim^F_d(k_0,\ldots,k_n)=k} v_{k_0}(C_0) \cdots v_{k_n}(C_n) ,
\end{align*}
where $(*)$ follows from the induction hypothesis.

The second case can be reduced to the first case by using De Morgan's Laws:
\begin{align*}
   F(X_0,\ldots,X_n) & = \neg \big(\neg F_1(X_0,\ldots,X_m) \wedge \neg F_2(X_{m+1},\ldots,X_n)\big)
\\ & = \neg \big(F_1'(X_0,\ldots,X_m) \wedge F_2'(X_{m+1},\ldots,X_n)\big) ,
\end{align*}
where the Boolean formulas $F_1'$ and $F_2'$ are obtained from $\neg F_1$ and $\neg F_2$ by ``pulling the negation all the way down to the variables'' using again De Morgan's Laws. Applying the first case, we obtain
\begin{align*}
  \Expect\bigg[ & v_k\Big( F\big(\vct Q_0 C_0,\ldots,\vct Q_n C_n\big) \Big)\bigg] = \Expect\bigg[ v_{d-k}\Big( \neg F\big(\vct Q_0 C_0,\ldots,\vct Q_n C_n\big) \Big)\bigg]
\\ & = \sum_{\dim^{\neg F}_d(k_0,\ldots,k_n)=d-k} v_{k_0}(C_0) \cdots v_{k_n}(C_n) = \sum_{\dim^F_d(k_0,\ldots,k_n)=k} v_{k_0}(C_0) \cdots v_{k_n}(C_n) . \qedhere
\end{align*}
\end{proof}

\bibliographystyle{myalpha}
\bibliography{statdim}

\appendix

\section{Steiner formulas}\label{sec:Steiner-form}

In the Euclidean case the Steiner formula describes the volume of the tubular neighborhood of a convex body as a polynomial in the radius with coefficients given by (rescaled) Euclidean intrinsic volumes. In a straightforward way one obtains a spherical version of this, which has no longer the exact form of a polynomial, but a form in which the monomials are replaced by the volume functions of tubes around subspheres.

Using the conic instead of the spherical viewpoint, one obtains very elegant Steiner formulas,~cf.~\cite{McC:thesis,MT:13}: Let $C\subseteq\IR^d$ be a closed convex cone, $M\in\hmB(\IR^d)$, and $\M\in\hmB(\IR^d,\IR^d)$. Then
\begin{align}
   \Prob\big\{ \|\Pi_C(\vct g)\|^2 \geq r\big\} & = \sum_{k=0}^d \Prob\big\{ \chi_k^2\geq r\big\}\; v_k(C) ,
\label{eq:Steiner-v}
\\ \Prob\big\{ \Pi_C(\vct g)\in M \text{ and } \|\Pi_C(\vct g)\|^2 \geq r\big\} & = \sum_{k=0}^d \Prob\big\{ \chi_k^2\geq r\big\}\; \Phi_k(C,M) ,
\label{eq:Steiner-Phi}
\\ \Prob\big\{ \bproj_C(\vct g)\in \M \text{ and } \|\Pi_C(\vct g)\|^2 \geq r\big\} & = \sum_{k=0}^d \Prob\big\{ \chi_k^2\geq r\big\}\; \Theta_k(C,\M) ,
\label{eq:Steiner-Theta}
\end{align}
where $\vct g\in\IR^d$ Gaussian and $\chi_k^2\in\IR$ denotes a chi-squared distributed random variable with $k$ degrees of freedom. Since $C$ is here allowed to be any closed convex cone, these formulas should be understood that in the nonpolyhedral case the left-hand sides can be expressed in the form given by the right-hand sides, and the intrinsic volumes $v_k(C)$, the curvature measures $\Phi_k(C,M)$ and the support measures $\Theta_k(C,\M)$ can be defined in this way. A proof for~\eqref{eq:Steiner-Theta} can be found in~\cite[Thm.~6.5.1]{SW:08}, and~\eqref{eq:Steiner-Phi} and~\eqref{eq:Steiner-v} follow of course from~\eqref{eq:Steiner-Theta}.

Using these formulas, one can prove that the support measures are additive and weakly continuous, as listed in Proposition~\ref{prop:props-supp-meas} under (6) and (7). Details for this proof can be found in~\cite[Thm.~6.5.2]{SW:08}. The following lemma follows directly from the fact that~\eqref{eq:Steiner-Theta} characterizes $\Theta_k(C,\M)$, $k=0,\ldots,d$, and from the fact that $\bproj_C(\vct x)\in\BL(C)$ for every $\vct x\in\IR^d$.

\begin{lemma}\label{lem:loc-suppmeas}
Let $C,D\in\P(\IR^d)$ and $\M\in\hmB(\IR^d,\IR^d)$. If $\M\subseteq\BL(C)\cap\BL(D)$, then $\Theta_k(C,\M)=\Theta_k(D,\M)$.
\end{lemma}

\section{Genericity}\label{sec:genericity}

If $L_0,\ldots,L_n\subseteq\IR^d$ are linear subspaces, then
\begin{equation}\label{eq:dim(L_1capL_2)}
  \dim(L_0\cap\cdots\cap L_n) \geq \max\{ 0, \dim(L_0)+\cdots +\dim(L_n) - nd\} .
\end{equation}
We say that the tuple $(L_0,\ldots,L_n)$ is in \emph{general position} if this inequality is an equality for every selection of some of these spaces, i.e., if
  \[ \dim(L_{i_0}\cap\cdots\cap L_{i_k}) = \max\{ 0, \dim(L_{i_0})+\cdots +\dim(L_{i_k}) - kd\} \quad \text{for all} \quad 0\leq i_0<i_1<\cdots<i_k\leq n . \]
Note that $(L_0,\ldots,L_n)$ is always in general position if $\dim(L_i)=0$ for some~$i$.

For polyhedral cones $C_0,\ldots,C_n\in\P(\IR^d)$ we say that $(C_0,\ldots,C_n)$ is in general position if $(L_0,\ldots,L_n)$ is in general position for all $L_i\in\mL(C_i)$, $i=0,\ldots,n$.

\begin{remark}
Recall that $\mL(F)\subseteq\mL(C)$ if $F$ is a face of $C$. Therefore, if $(C_0,\ldots,C_n)$ is in general position and $F_i$ is a face of $C_i$, $i=0,\ldots,n$, then $(F_0,\ldots,F_n)$ is in general position as well.
\end{remark}

\begin{lemma}\label{lem:TQC-alm-sure-generic}
Let $C_0,\ldots,C_n\in\P(\IR^d)$ and $\vct T_0,\ldots,\vct T_n\in\Gl_d$, and define
\begin{align}\label{eq:def-Q_T(C_0,..,C_n)}
  \Q_{\vct T}(C_0,\ldots,C_n) & := \big\{(\vct Q_0,\ldots,\vct Q_n)\in O(d)^{n+1}\mid (\vct T_0\vct Q_0C_0,\ldots,\vct T_n\vct Q_nC_n) \text{ is in general position}\big\} .
\end{align}
Then $\Q_{\vct T}(C_0,\ldots,C_n)$ is an open and dense subsets of $O(d)^{n+1}$.
\end{lemma}

\begin{proof}
We can write $\Q_{\vct T}(C_0,\ldots,C_n)$ as the finite intersection
\begin{align*}
   \Q_{\vct T}(C_0,\ldots,C_n) & = \bigcap_{L_0\in\mL(C_0)} \cdots \bigcap_{L_n\in\mL(C_n)} \Q_{\vct T}(L_0,\ldots,L_n) ,
\end{align*}
so it suffices to show the claim for linear subspaces $C_i=L_i$, $0\leq i\leq n$. Furthermore, using polarity (and replacing $L_i$ by $L_i^\bot$), we can reformulate the claim so that we need to show
  \[ \dim(\vct T_0^\polar\vct Q_0L_0+\cdots+\vct T_n^\polar\vct Q_nL_n) = \min\{d,d_1+\cdots+d_n\} ,\qquad d_i:=\dim(L_i) \]
for all $(\vct Q_0,\ldots,\vct Q_n)$ in some open dense subset of $O(d)^{n+1}$. By orthogonal invariance we may assume without loss of generality that $L_i=\IR^{d_i}\times\{\vct0\}$. Interpreting $\vct T_i^\polar\vct Q_i\in\IR^{d\times d}$ as a matrix and defining $\vct A_i\in\IR^{d\times d_i}$ to consist of the first $d_i$ columns of~$\vct T_i^\polar\vct Q_i$, the claim thus becomes that the matrix $\begin{pmatrix} \vct A_0 & \cdots & \vct A_n\end{pmatrix}$ has full rank. This rank condition can be expressed by the nonvanishing of the Gram determinant, and it is readily checked that this determinant is nonzero for an open dense subset of~$O(d)^{n+1}$.
\end{proof}

\begin{lemma}\label{lem:generic-skeleton-inters}
Let $(C_0,\ldots,C_n)$, $C_i\in\P(\IR^d)$, be in general position and let $k>0$. Then the $k$-skeleton of the intersection $C_0\cap\cdots\cap C_n$ is given by the disjoint union
\begin{equation}\label{eq:generic-skeleton-inters}
  \Skel_k(C_0\cap\cdots\cap C_n) = \Disjunionu{k_0+\cdots+k_n=k+nd} \big( \Skel_{k_0}(C_0)\cap\cdots\cap\Skel_{k_n}(C_n)\big)
\end{equation}
\end{lemma}

\begin{proof}
We abbreviate $C:=C_0\cap\cdots\cap C_n$. The genericity condition (and the assumption $k>0$) implies that the linear span of every $k$-dimensional face of~$C$, $L\in\mL_k(C)$, can be written in the form $L=L_0\cap\cdots\cap L_n$ with $L_i\in\mL_{k_i}(C_i)$, $0\leq i\leq n$, and $k_0+\cdots+k_n=k+nd$. Since in this case we have
  \[ \inter_L(C\cap L) = \inter_{L_0}(C_0\cap L_0)\cap\cdots\cap \inter_{L_n}(C_n\cap L_n) \subseteq \Skel_{k_0}(C_0)\cap\cdots\cap\Skel_{k_n}(C_n) , \]
it follows the inequality `$\subseteq$' in~\eqref{eq:generic-skeleton-inters}. On the other hand, if $L_i\in\mL_{k_i}(C_i)$, $0\leq i\leq n$, such that $L_0\cap\cdots\cap L_n\cap C\neq\emptyset$, then $\dim(L_0\cap\cdots\cap L_n)=k_0+\cdots+k_n-nd$, which shows the reverse inclusion.
\end{proof}

\begin{lemma}\label{lem:gen-pos->loc-cont}
Let $C_0,\ldots,C_n\in\P(\IR^d)$ and $\vct T_0,\ldots,\vct T_n\in\Gl_d$ such that $(\vct T_0C_0,\ldots,\vct T_nC_n)$ is in general position, and let $k>0$. Furthermore, let $M_0,\ldots,M_n\in\hmB(\IR^d)$, $\M_0,\ldots,\M_n\in\hmB(\IR^d,\IR^d)$, and let $\vp_1,\vp_2 \colon O(d)^{n+1} \to \IR$ be defined by
\begin{align*}
   \vp_1(\vct Q_0,\ldots,\vct Q_n) & = \Psi_k\bigg( \bigcap_{i=0}^n \vct T_i \vct Q_i C_i , \bigcap_{i=0}^n \vct T_i \vct Q_i M_i\bigg) ,
\\ \vp_2(\vct Q_0,\ldots,\vct Q_n) & = \Theta_k\bigg( \bigcap_{i=0}^n \vct T_i \vct Q_i C_i , \bigwedge_{i=0}^n \vct T_i\vct Q_i \big( \BL(C_i)\cap \M_i\big)\bigg) .
\end{align*}
Then $\phi_1$ and $\phi_2$ are locally continuous around $(\Id_d,\ldots,\Id_d)$.
\end{lemma}

\begin{proof}
Since $(\vct T_0C_0,\ldots,\vct T_nC_n)$ is in general position, $(\Id_d,\ldots,\Id_d)\in\Q_{\vct T}(C_0,\ldots,C_n)$, cf.~\eqref{eq:def-Q_T(C_0,..,C_n)}, and by Lemma~\ref{lem:TQC-alm-sure-generic} it follows that there exists an open ball $\mU\subseteq\Q_{\vct T}(C_0,\ldots,C_n)$ around $(\Id_d,\ldots,\Id_d)$. As seen in Lemma~\ref{lem:generic-skeleton-inters}, all supporting subspaces $L\in\mL(C_{\vct Q})$, $C_{\vct Q}:=\bigcap_{i=0}^n \vct T_i\vct Q_i C_i$, are of the form $L=\bigcap_{i=0}^n \vct T_i\vct Q_i L_i$ with $L_i\in\mL(C_i)$ and $\sum_{i=0}^n \dim(L_i)=\dim(L)+nd$. Moreover, in the open ball $\mU\subseteq\Q_{\vct T}(C_0,\ldots,C_n)$ the set
  \[ \mL_{\vct Q} := \Big\{ (L_0,\ldots,L_n)\in\mL(C_0)\times\cdots\times\mL(C_n) \mid \bigcap_{i=0}^n \vct T_i\vct Q_i L_i\in\mL(C_{\vct Q})\Big\} \]
is constant; otherwise there would exist $(\vct Q_0,\ldots,\vct Q_n)\in \mU$ such that $(\vct T_0\vct Q_0C_0,\ldots,\vct T_n\vct Q_nC_n)$ is not in general position. In the neighborhood~$\mU$ of~$(\Id_d,\ldots,\Id_d)$ we thus have (denoting $\mL_\mU:=\mL_{\vct Q}$)
  \[ \phi_1(\vct Q_0,\ldots,\vct Q_n) = \sum_{(L_0,\ldots,L_n)\in\mL_\mU} \gamma_L\bigg( \bigcap_{i=0}^n \vct T_i\vct Q_i(C_i\cap L_i\cap M_i)\bigg) . \]
This is easily seen to depend continuously on the $\vct Q_i$, so that $\phi_1$ is a locally continuous function.

Analogously, the locally constant face structure of $C_{\vct Q}$ implies the local continuity of~$\phi_2$.
\end{proof}

\begin{proposition}\label{prop:genericity}
Let $C_0,\ldots,C_n\in\P(\IR^d)$, $M_0,\ldots,M_n\in\hmB(\IR^d)$, $\M_0,\ldots,\M_n\in\hmB(\IR^d,\IR^d)$, and let $\vct T_0,\ldots,\vct T_n\in\Gl_d$ and $k>0$. Then for $\vct Q_0,\ldots,\vct Q_n\in O(d)$ iid uniformly at random the following holds:
\begin{enumerate}
  \item The expectations
    \begin{align*}
       & \Expect\bigg[ \Psi_k\bigg( \bigcap_{i=0}^n \vct T_i \vct Q_i C_i , \bigcap_{i=0}^n \vct T_i \vct Q_i M_i\bigg)\bigg] , & & \Expect\bigg[ \Theta_k\bigg( \bigcap_{i=0}^n \vct T_i \vct Q_i C_i , \bigwedge_{i=0}^n \vct T_i \vct Q_i\big( \BL(C_i)\cap \M_i\big)\bigg)\bigg]
    \end{align*}
        exist and are finite.
  \item If $\Skel_j(C_0)\cap M_0=\Skel_j(\tilde C_0)\cap M_0$ for all $j=0,\ldots,d$, $\tilde C_0\in\P(\IR^d)$, then almost surely
    \[ \Psi_k\bigg( \bigcap_{i=0}^n \vct T_i \vct Q_i C_i , \bigcap_{i=0}^n \vct T_i \vct Q_i M_i\bigg) = \Psi_k\bigg( \bigcap_{i=0}^n \vct T_i \vct Q_i \tilde C_i , \bigcap_{i=0}^n \vct T_i \vct Q_i M_i\bigg) , \]
        where $\tilde C_i:=C_i$ for $i=1,\ldots,n$.
\end{enumerate}
\end{proposition}

\begin{proof}
(1) Combining Lemma~\ref{lem:gen-pos->loc-cont} and Lemma~\ref{lem:TQC-alm-sure-generic} with the fact that $O(d)^{n+1}$ is compact yields the existence of the expectations. The expectations are finite, since $\Theta_k(C,\M)\leq 1$, and $\Psi_k(C,M)$ can be upper bounded, for example, by $\Psi_k(C,M)\leq u_k(C)\leq f_k(C)$, and $f_k(C_0\cap\cdots\cap C_n)\leq \prod_{i=0}^n \sum_{j=0}^d f_j(C_i)$.

(2) To simplify notation let $C := \bigcap_{i=0}^n \vct T_i \vct Q_i C_i$, $\tilde C := \bigcap_{i=0}^n \vct T_i \vct Q_i \tilde C_i$, $M := \bigcap_{i=0}^n \vct T_i \vct Q_i M_i$.
Combining Lemma~\ref{lem:TQC-alm-sure-generic} and Lemma~\ref{lem:generic-skeleton-inters} implies that almost surely
\begin{align*}
   M\cap \Skel_k(C) & = M\cap \bigcup_{k_0+\cdots+k_n=k+nd} \bigg( \bigcap_{i=0}^n \Skel_{k_i}(\vct T_i \vct Q_i C_i)\bigg) = \bigcup_{k_0+\cdots+k_n=k+nd} \bigg( \bigcap_{i=0}^n \vct T_i \vct Q_i \big( M_i\cap \Skel_{k_i}(C_i)\big) \bigg)
\\ & = \bigcup_{k_0+\cdots+k_n=k+nd} \bigg( \bigcap_{i=0}^n \vct T_i \vct Q_i \big( M_i\cap \Skel_{k_i}(\tilde C_i)\big) \bigg) = M\cap \bigcup_{k_0+\cdots+k_n=k+nd} \bigg( \bigcap_{i=0}^n \Skel_{k_i}(\vct T_i \vct Q_i \tilde C_i)\bigg)
\\ & = M\cap \Skel_k(\tilde C) .
\end{align*}
Therefore, the locality of the polyhedral measures, cf.~Proposition~\ref{prop:props-polyh-curv-meas}(3), implies that almost surely $\Psi_k(C,M)=\Psi_k(\tilde C,M)$.
\end{proof}

\begin{proposition}\label{prop:genericity_BL}
Let $C,D\in\P(\IR^d)$, $\M,\N_j\in\hmB(\IR^d,\IR^d)$, $j=1,2,\ldots$, such that $\N_j\subseteq\BL(C)$ for all~$j$, $\N_j\cap \N_{j'}=\emptyset$ for all $j\neq j'$, and $\M\subseteq\BL(D)$. Furthermore, let $\vct T\in\Gl_d$ and $k>0$. Then for almost all~$\vct Q\in O(d)$,
\begin{equation}\label{eq:add-Theta_k-M}
  \Theta_k\bigg( \vct{TQ} C\cap D , \bigg(\bigcup_{j=1}^\infty \vct{TQ} \N_j\bigg)\wedge\M\bigg) = \sum_{j=1}^\infty \Theta_k\big( \vct{TQ} C\cap D , \vct{TQ}\N_j\wedge\M\big) .
\end{equation}
\end{proposition}

\begin{proof}
By Lemma~\ref{lem:TQC-alm-sure-generic} we have that $(\vct{TQ}C,D)$ is almost surely in general position. So assuming that $(C,D)$ is in general position, what we will do in the rest of the proof, it suffices to show that~\eqref{eq:add-Theta_k-M} holds for $\vct T=\vct Q=\Id_d$. Using Proposition~\ref{prop:biconic-conj} and Lemma~\ref{lem:decomp-Theta_k(C,M)}, we have
\begin{align*}
   \Theta_k\bigg( C\cap D , \bigg(\bigcup_{j=1}^\infty \N_j\bigg)\wedge\M\bigg) & = \sum_{L\in\mL_k(C\cap D)} \Theta_k\bigg( L , \bigcup_{j=1}^\infty \big(\N_j\wedge\M\big) \cap \BL(L) \bigg)
\\ & = \sum_{\substack{L_0\in\mL_i(C), L_1\in\mL_j(D)\\ i+j=k+d}} \Theta_k\bigg( L_0\cap L_1 , \bigcup_{j=1}^\infty \big(\N_j\wedge\M\big) \cap \big(\BL(L_0)\wedge \BL(L_1)\big) \bigg) ,
\end{align*}
where the second equality follows from the assumption that $(C,D)$ are in general position. Decomposing $\N_j=\bigcup_{L_0\in\mL(C)}(\N_j\cap \BL(L_0))$ and $\M=\bigcup_{L_1\in\mL(C)}(\M\cap \BL(L_1))$, and arguing as in Lemma~\ref{lem:decomp-Theta_k(C,M)},
\begin{align*}
   & \Theta_k\bigg( L_0\cap L_1 , \bigcup_{j=1}^\infty \big(\N_j\wedge\M\big) \cap \big(\BL(L_0)\wedge \BL(L_1)\big) \bigg)
\\ & = \sum_{\substack{L_0'\in\mL_i(C), L_1'\in\mL_j(D)\\ i+j=k+d}} \Theta_k\bigg( L_0\cap L_1 , \bigcup_{j=1}^\infty \Big(\big(\N_j\cap \BL(L_0')\big)\wedge\big(\M\cap \BL(L_1')\big)\Big) \cap \big(\BL(L_0)\wedge \BL(L_1)\big) \bigg)
\\ & = \Theta_k\bigg( L_0\cap L_1 , \bigcup_{j=1}^\infty \big(\N_j\cap \BL(L_0)\big)\wedge\big(\M\cap \BL(L_1)\big) \bigg) = \gamma_d\Bigg( \add\bigg( \bigcup_{j=1}^\infty \tilde\N_j \bigg) \Bigg) ,
\end{align*}
where $\tilde\N_j := \big(\N_j\cap \BL(L_0)\big)\wedge\big(\M\cap \BL(L_1)\big)$ and $\add\colon\IR^{d+d}\to\IR^d$ is given by $\add(\vct x,\vct x') = \vct x+\vct x'$, cf.~\eqref{eq:addoPi_C=Id}. Since $\IR^d=(L_0\cap L_1) + L_0^\bot + L_1^\bot$ with $\dim(L_0\cap L_1) + \dim L_0^\bot + \dim L_1^\bot = d$, every element $\vct z\in\IR^d$ can be uniquely written as $\vct z=\vct x+\vct y_0+\vct y_1$ with $\vct x\in (L_0\cap L_1)$, $\vct y_0\in L_0^\bot$, $\vct y_1\in L_1^\bot$. This shows that $\tilde\N_j$ are mutually disjoint as $\N_j$ are, which implies
\begin{align*}
   \gamma_d\Bigg( \add\bigg( \bigcup_{j=1}^\infty \tilde\N_j \bigg) \Bigg) & = \sum_{j=1}^\infty \gamma_d\big( \add( \tilde\N_j ) \big) .
\end{align*}
Putting things together, we obtain
\begin{align*}
   \Theta_k\bigg( C\cap D , \bigg(\bigcup_{j=1}^\infty \N_j\bigg)\wedge\M\bigg) & = \sum_{\substack{L_0\in\mL_i(C), L_1\in\mL_j(D)\\ i+j=k+d}} \; \sum_{j=1}^\infty\Theta_k\Big( L_0\cap L_1 , \big(\N_j\cap \BL(L_0)\big)\wedge\big(\M\cap \BL(L_1)\big) \Big)
\\ & = \sum_{j=1}^\infty \Theta_k\big( C\cap D , \N_j\wedge\M\big) . \qedhere
\end{align*}
\end{proof}

\end{document}